\documentclass[12pt]{amsart}
\usepackage{amsmath}
\usepackage{amssymb}
\usepackage{amscd}
\usepackage{graphicx,color}
\usepackage{bm}
\newtheorem{theorem}{Theorem}[section]
\newtheorem{lemma}[theorem]{Lemma}
\newtheorem{proposition}[theorem]{Proposition}
\newtheorem{corollary}[theorem]{Corollary}

\theoremstyle{definition}
\newtheorem{definition}[theorem]{Definition}

\numberwithin{equation}{section}

\newtheorem{remark}[theorem]{Remark}

\title[Satake's Good basic invariants]{Satake's good basic invariants for finite complex reflection groups}
\author{Yukiko Konishi}
\address{ 
Department of Mathematics, 
College of Liberal Arts,
Tsuda University,
Tokyo 187-8577, Japan 
}
\email{konishi@tsuda.ac.jp}
\author{Satoshi Minabe}
\address{Department of Mathematics, 
Tokyo Denki University, Tokyo 120-8551,  Japan}
\email{minabe@mail.dendai.ac.jp}
\begin{document}
\maketitle
\begin{abstract}
In  \cite{Satake2020} Satake introduced
the notions of admissible triplets and good basic invariants 
for  finite complex reflection groups.
 For irreducible finite Coxeter groups,
he showed the existence and the uniqueness of 
good basic invariants.
Moreover he showed that good basic invariants are flat in the sense of K.\,Saito's flat structure. He also obtained a formula for the multiplication of the Frobenius structure.
In this article, we generalize his  results to  finite complex  reflection groups.
We first study  the existence and the uniqueness of 
good basic invariants.
Then for duality groups, we show that good basic invariants are flat 
in the sense of the natural Saito structure constructed in \cite{KMS2018}.
We also give a formula for the potential vector fields of the multiplication in terms of 
the good basic invariants. Moreover, in the case of irreducible finite Coxeter groups, 
we derive a formula for the potential functions of the associated Frobenius manifolds.  
\end{abstract}

%%%%%%%%%%%%%
\section{Introduction}
%%%%%%%%%%%%%
In 1980, Saito--Sekiguchi--Yano found   ``a flat generator system'' on the orbit
spaces of  irreducible finite Coxeter groups\footnote{We call a finite group
generated by reflections  acting on a real vector space $V_{\mathbb{R}}$
a finite Coxeter group. By extending the action to  $V=V_{\mathbb{R}}\otimes \mathbb{C}$, 
an irreducible finite Coxeter group can be regarded as a finite complex reflection group.}
\cite{SaitoSekiguchiYano}. See also \cite{Saito1993}.
Later Dubrovin \cite{Dubrovin1993-2} invented a formulation called Frobenius manifold -- sometimes 
called Frobenius structure --
which includes Saito--Sekiguchi--Yano's flat structure. See also \cite{Dubrovin1998}.
Since then the Frobenius structure is constructed for
a wide range of reflection groups.
Recently,  Saito structure without metric   defined in \cite{Sabbah}
is constructed for
duality groups
\cite{KatoManoSekiguchi2015}
\cite{Arsie-Lorenzoni2016} \cite{KMS2018}.
We remind the reader that
a duality group is an irreducible finite complex reflection group satisfying certain  conditions,
and that an irreducible  finite Coxeter  group  is  a duality group.

Let $G$ be a duality group acting on the vector space $V=\mathbb{C}^n$.
The Saito structure without metric for $G$ obtained in \cite{KMS2018} is called 
the natural Saito structure.
Roughly speaking,
it consists of a certain torsion-free, flat connection and the multiplication on the holomorphic tangent  bundle of the orbit space.
It can be characterized as a structure ``dual'' to 
a certain structure given by the trivial  connection on the holomorphic tangent bundle $TV$ of $V$.
 However,  this characterization is somewhat roundabout
 and  a more straightforward characterization  
 is desired. 
The notion of good basic invariants proposed by Satake \cite{Satake2020}  gives a solution to this problem.

In \cite{Satake2020}, Satake  gave a definition of
 an admissible triplet and good basic invariants for finite complex reflection groups. 
 {\it An admissible triplet} consists of a primitive $d_1$-th root of unity $\zeta$,
$\zeta$-regular element $g\in G$, and a regular vector $q$ such that
$gq=\zeta q$.  (Here $d_1$ is the highest degree of $G$.)
{\it A  $g$ homogeneous basis} $z$ is a linear coordinate system of $V$ consisting of eigenvectors of the action of $g$ on $V^*$.
A set of basic invariants is said to be {\it good} with respect to an admissible triplet $(g,\zeta, q)$ if its certain derivatives by $z$ vanish at $q$.
In the case of irreducible  finite Coxeter groups,
Satake showed the existence and the uniqueness of good basic invariants.
He also showed the independence of good basic invariants
of the choice of admissible triplet.
(Precisely speaking,  the vector subspace spanned by good basic invariants
is unique and independent.) 
Moreover, in relation to the Frobenius structure, 
he showed that a set of good basic invariants is a flat generator system
and he obtained  an explicit expression of the structure constants
of the multiplication.

In this article, we study  Satake's good basic invariants 
for finite complex reflection groups.
The article consists of three parts and is organized as follows.

The first part (\S \ref{sec:preliminary}--\S\ref{sec:good}) concerns the goodness of basic invariants.
\S \ref{sec:preliminary} is a preliminary.
In \S \ref{sec:regularelements}, first we recall the definition and properties of regular elements. Then we introduce Satake's definition of admissible triplet.
For finite complex reflection groups, the necessary and sufficient condition for the existence of regular elements is known (Theorem \ref{thm:criteria}).
Using the condition,
we list all irreducible finite complex reflection groups which admit
admissible triplets.
In \S \ref{sec:good}, we first recall Satake's definition and properties of 
a $g$-homogeneous basis,  compatible basic invariants, and 
good basic invariants.
We use 
the name  {\it a $(g,\zeta)$-graded coordinate system} instead of a $g$-homogeneous basis
to specify the choice of $\zeta$.
 (It  sometime happens that a $\zeta$-regular element
 $g$ is also $\zeta'$-regular for another primitive $d_1$-th  root of unity
 and that a $g$-homogeneous basis depends on the choice of $\zeta$.)
Then we show the existence of a set of good basic invariants
and the uniqueness
of its span  $\mathcal{I}$ in the ring of 
$G$-invariant polynomials $S[V]^G$
for a given admissible triplet
(Theorems \ref{existence-good}, \ref{uniqueness-good}).
We  also show that $\mathcal{I}$ is independent of the choice of admissible triplet when the highest degree $d_1$ is isolated
(Proposition \ref{indep}, Theorem \ref{unique}, Remark \ref{ad=2}).

The second part (\S \ref{second-part}--\S \ref{Satake-C}) of this article concerns the relationship with
 the natural Saito structure for duality groups studied in \cite{KMS2018}.
 In \S \ref{second-part},
 we collect necessary definitions and  properties from \cite{KMS2018}. 
 In \S \ref{sec:duality-groups}, we first show that
 good basic invariants are flat in the sense of the natural Saito structure
 (Theorem \ref{good-implies-flat}).
 Then we derive an expression of the structure constants of the multiplication in terms of
 good basic invariants $x$ and $(g,\zeta)$-graded coordinates $z$
 (Theorem \ref{main-theorem2}, Corollary \ref{vector-potential}).
 These results show that the natural Saito structure
for the duality groups is fully reconstructed 
by a set of good basic invariants.  
 In \S \ref{Satake-C}, we study the case  of irreducible finite Coxeter groups. 
 We obtain an expression of the potential function (Theorem \ref{thm:potential})
 in terms of good basic invariants. Hence the Frobenius structure for these groups
is also reconstructed by a set of good basic invariants.  
We also show that our expression for the multiplication agrees with that of Satake's
(Proposition \ref{C-Satake}). 

The third part (\S \ref{sec:examples}--\S \ref{sec:examples2})
contains some examples of admissible triplets, good basic invariants, 
and potential vector fields. 
\S \ref{sec:examples} treats exceptional groups of 
rank $2$.  \S \ref{sec:examples2} treats $G(m, m, n)$
and $G(m,1,n)$ for $m\geq 2$ and $n\leq 4$.
%Based on a natural correspondence between admissible triplets, 
%we show that the natural Saito structure of $G(m,1,n)$ is 
%obtained by projecting that of $G(m,m,n+1)$
%onto some linear hyperplane (Theorem \ref{thm:projection}, Proposition \ref{prop:projection}).

%%%%%%%%%%%%%%%%%
\subsection*{Acknowledgment}
%%%%%%%%%%%%%%%%%%
The work of Y.K.  is supported in part by
KAKENHI Kiban-S (16H06337) and KAKENHI Kiban-S (21H04994).
S.M. is supported in part by
KAKENHI Kiban-C (23K03099). 
Y.K. and S.M. thank Ikuo Satake for valuable discussions and comments.
The authors also thank the referee  for helpful comments.

%%%%%%%%%%%%%
\section{Preliminary I}
\label{sec:preliminary}
%%%%%%%%%%%%%
%%%%%%%%%%%%%
\subsection{Notations}
%%%%%%%%%%%%%
For an integer $i$,  $\mathbb{Z}_{\geq i}$ denotes the set of integers greater than or equal to $i$.
Let $n$ be a positive integer.  
\begin{itemize}
\item For $a=(a_1,\ldots, a_n)\in \mathbb{Z}_{\geq 0}^n$
and $d=(d_1,\ldots, d_n)\in \mathbb{Z}_{\geq 0}^n$, 
$$
|a|=a_1+\cdots+a_n~,\quad a!=a_1! \cdots a_n!~,
\quad 
a\cdot d=a_1d_1+\cdots +a_n d_n~.
$$
\item  For $a=(a_1,\ldots, a_n)\in \mathbb{Z}_{\geq 0}^n$, 
and a set of $n$ variables  or a set of $n$ complex numbers $z=(z^1,\ldots, z^n)$,
$$
z^a=(z^1)^{a_1}\cdots (z^n)^{a_n}~.
$$
\item For $a=(a_1,\ldots, a_n)\in \mathbb{Z}_{\geq 0}^n$, 
and a set of $n$ variables  $z=(z^1,\ldots, z^n)$,
$$
\frac{\partial^a}{\partial z^a}=\left(\frac{\partial}{\partial z^1}\right)^{a_1}
\cdots \left(\frac{\partial}{\partial z^n}\right)^{a_n}~.
$$
\item $\bm{e}_{\alpha}$ ($1\leq \alpha\leq n$) denotes 
$$
\bm{e}_{\alpha}=(0,\ldots, \overbrace{1}^{\text{$\alpha$-th}},\ldots, 0) \in \mathbb{Z}_{\geq 0}^n~.
$$
\end{itemize}

We write Kronecker's  delta as
\begin{equation}\nonumber 
\delta^{\gamma}_{\beta}=\begin{cases}
1&(\beta=\gamma)\\
0&(\beta\neq \gamma)
\end{cases}~.
\end{equation}
(Superscripts and subscripts will be used in accordance with the
covariance/contravariance of tensor components.)

The identity matrix is denoted $I$, and the zero matrix is denoted $O$.

%%%%%%%%%%%%%%%%%%%%%%%%%
\subsection{Coordinate systems and bases}
%%%%%%%%%%%%%%%%%%%%%%%%%%
Let $n$ be a positive integer and $V=\mathbb{C}^n$. 
The action of 
$g\in GL(V)$ on the ring  
$S[V]$ of polynomial functions  on $V$ is defined by
\begin{equation}\nonumber 
(g^*F)(v)=F(g^{-1}v)\quad (F\in S[V], ~~v\in V)~.
\end{equation}

For a diagonalizable matrix  $g\in GL(V)$,
there exists a basis $q_1,\ldots, q_n$ of $V$ consisting of 
eigenvectors of $g$. 
Then the associated linear coordinate system $z^1,\ldots,  z^n $ of $V$ (i.e. the basis of $V^*$ satisfying $z^{\alpha}(q_{\beta})=\delta^{\alpha}_{\beta}$) is a  basis of $V^*$ consisting of eigenvectors of the $g$-action on $V^*$.  More precisely, if $gq_{\alpha}=\lambda_{\alpha}q_{\alpha}$ ($1\leq \alpha\leq n$), 
\begin{equation}\label{g-action-z}
g^*z^{\alpha}=\lambda_{\alpha}^{-1}z^{\alpha}~.
\end{equation}

\begin{lemma} \label{dFdz}
Let $g\in GL(V)$ be a diagonalizable matrix with eigenvalues
$\lambda_1,\ldots,\lambda_n$ and 
let $z=(z^1,\ldots, z^n)$ be a linear coordinate system of $V$
satisfying  \eqref{g-action-z}.
Then for 
a homogeneous polynomial $F\in S[V]$ satisfying $g^*F=F$ and an eigenvector $q$
of $g$ belonging to the eigenvalue $\lambda_1$, 
the following holds:
$$\text{if}\quad
\frac{\partial^{a} F}{\partial z^{a}}(q)\neq 0  \quad
\text{then}\quad 
\lambda_1^{\mathrm{deg}\,F-|a|}\lambda^{a}=1~,
$$
where $a\in \mathbb{Z}^n_{\geq 0}$.
\end{lemma}

\begin{proof}

Any homogeneous polynomial $F\in S[V]$ is expressed as
\begin{equation}\nonumber
F=\sum_{\begin{subarray}{c}b\in \mathbb{Z}_{\geq 0}^n;\\
|b|=\mathrm{deg}\,F
\end{subarray}}
A_{b}z^{b} \quad (A_{b}\in \mathbb{C})~.
\end{equation}
By \eqref{g-action-z}, we have $g^*z^b=\lambda^{-b} z^b$.
So the assumption $g^*F=F$ implies that 
$A_b=0$ if $\lambda^{b}\neq 1$. 
Thus a homogeneous polynomial $F$ satisfying $g^*F=F$ is expressed
as follows:
 \begin{equation}\nonumber 
 F=\sum_{\begin{subarray}{c}b\in \mathbb{Z}_{\geq 0}^n;\\
|b|=\mathrm{deg}\,F,
\\
\lambda^{b}=1
\end{subarray}}
A_{b}z^{b}\quad (A_{b}\in \mathbb{C})~.
\end{equation}

Now let $k$ be the dimension of the $\lambda_1$-eigenspace of $g$.
Renumbering  the eigenvalues  if necessary,
we assume that the first $k$ eigenvalues of $g$ are $\lambda_1$.
Let $q_{\alpha}$($1\leq \alpha\leq n$) be the basis of $V$ associated to 
$z=(z^1,\ldots, z^n)$, 
i.e. it is the basis satisfying 
\begin{equation}\label{z-q}
z^{\alpha}(q_{\beta})=\delta^{\alpha}_{\beta}~.
\end{equation}
Then the eigenvector $q$ is a linear combination of $q_1,\ldots, q_k$.

First we show the statement when $k=1$. In this case,
$q$ is a constant multiple of $q_1$.
Given the relation \eqref{z-q},
$\frac{\partial^a z^b}{\partial z^a}(q_1)\neq 0$ holds if and only if
$\frac{\partial^a {z}^b}{\partial {z}^a}$ is a nonzero multiple of $z_1^k$.
Therefore if we differentiate $F$ by ${z}^a$ and evaluate it at $q_1$,
 terms that do not vanish has the following $b$:
\begin{equation}\nonumber
b=a+k\bm{e}_1~,\quad
\mathrm{deg}\,F=|b|=|a|+k~,
\quad 
\lambda^b=\lambda_1^k\lambda^a=1~.
\end{equation}
This proves the lemma when $k=1$.
Notice that the above argument works  when $q$ is a constant multiple of $q_1$
even if $k>1$.

Next we show the case when $k>1$. 
We take a  basis $\tilde{q}_1,\ldots,\tilde{q}_k$ of 
the $\lambda_1$-eigenspace of $g$ consisting of eigenvectors
such that $\tilde{q}_1=q$. 
Let $\tilde{z}^1,\ldots,\tilde{z}^n$ be
the linear coordinates of $V$ associated to the basis $\tilde{q}_1,\ldots,\tilde{q}_k, q_{k+1},\ldots,q_n$.
Then $\tilde{z}^1,\ldots,\tilde{z}^k$ are linear combination of 
$z^1,\ldots, z^k$ and $\tilde{z}_{\alpha}=z_{\alpha}$ for $k+1\leq \alpha\leq n$.
 Hence, by the chain rule, we have
$$
\frac{\partial^a}{\partial z^a}=\sum_{\begin{subarray}{c}
b\in \mathbb{Z}_{\geq 0}^n;\\|b|=|a|,\lambda^b=\lambda^a
\end{subarray}}
B_{a,b}
\frac{\partial^b}{\partial \tilde{z}^b}
$$
where $B_{a,b}$ are some constants.
This implies that if 
$\frac{\partial^a F}{\partial z^a}(q)\neq 0,$
then 
$\frac{\partial^b F}{\partial \tilde{z}^b}(q)\neq 0$
for at least one of the $b$'s satisfying $|b|=|a|$ and $\lambda^b=\lambda^a$.
But by the above argument,
$\frac{\partial^b F}{\partial \tilde{z}^b}(q)\neq 0$ implies that
$\lambda_1^{\mathrm{deg}F-|b| }\lambda^b=1$.
\end{proof}

%%%%%%%%%%%%%%%%%%%%%%%%%%%%
\section{Regular elements and admissible triplet}
\label{sec:regularelements}
%%%%%%%%%%%%%%%%%%%%%%%%%%%%
From here on, 
$G\subset GL(V)$ is a finite complex reflection group acting on $V=\mathbb{C}^n$.
The degrees of $G$ are denoted $d_1,d_2,\ldots , d_n$
and assumed to be in the descending order : $d_1\geq d_2\geq \ldots \geq d_n$.
The codegrees of $G$ are denoted $d_1^*(=0),d_2^*,\ldots, d_n^*$
and assumed to be in the ascending order : $d_1^*\leq d_2^*\leq \ldots \leq d_n^*$.
We also assume that any set of basic invariants $x^1,\cdots, x^n\in S[V]^G$ is
taken so that $\mathrm{deg}\, x^{\alpha}=d_{\alpha}$.

%%%%%%%%%%%%%%%%%%
\subsection{Regular elements}
%%%%%%%%%%%%%%%%%%
In this subsection, we recall  the definitions of regular vector, regular element
and the existence theorem for regular elements \cite{Springer} 
\cite[Ch.11]{LehrerTaylor}.

For $g\in G$ and $\zeta \in \mathbb{C}$, $V(g, \zeta)$ denotes
the $\zeta$-eigenspace of $g$ on $V$. 
\begin{definition}
\begin{enumerate}
\item[(i)] A vector $q\in V$ is {\it regular} if it lies on no reflection hyperplanes of $G$.
\item[(ii)] An element $g\in G$ is {\it regular} if $g$ has an eigenspace 
$V(g, \zeta)$ which contains a regular vector. In this case, 
we say that $g$ is a $\zeta$-regular element.
If $\zeta$ is a primitive $d$-th root of unity,
then we also say that $g$ is a $d$-regular element. 
\end{enumerate}
\end{definition}

\begin{remark}
A Coxeter element of an irreducible finite Coxeter group $G$  is a
$d_1$-regular element of $G$.
\end{remark}

\begin{remark} \label{remark:conjugacy}
Let $\zeta$ be a primitive $d$-th root of unity. 
Assume that a $\zeta$-regular element of $G$ exists. 
Then the followings are known.
See \cite[Theorem 11.24, Corollary 11.25]{LehrerTaylor}.
\begin{enumerate}
\item[(i)] Any $\zeta$-regular element is of order $d$. 
The order of its centralizer in $G$ equals to 
$\prod_{i:\, d_i\equiv 0\mod d} d_i$. 
Hence its conjugacy class consists of the $\prod_{i=1}^n d_i {\large /} \prod_{i: d_i\equiv 0 \mod d} d_i$ elements.
\item[(ii)] All  $\zeta$-regular elements are conjugate to each other in $G$.
\end{enumerate}
\end{remark}

For a positive integer $d$, we set
\begin{equation}\nonumber
\mathfrak{a}(d):=\#\{d_{\alpha}\mid \text{$d$ divides $d_{\alpha}$}\}
~,\quad
\mathfrak{b}(d):=\#\{d_{\alpha}^*\mid \text{$d$ divides $d_{\alpha}^*$}\}.
\end{equation}
Notice that $\mathfrak{b}(d)\geq 1$ holds since $d^*_1=0$ for any finite complex reflection group $G$. Moreover,
it is known that $\mathfrak{a}(d)\leq \mathfrak{b}(d)$ (cf. \cite[Theorem 11.28]{LehrerTaylor}). 
The following criterion for regularity is useful. 
See \cite[Theorem 11.28]{LehrerTaylor}

\begin{theorem}\label{thm:criteria}
Let $d$ be a positive integer.
A $d$-regular element of $G$ exists if and only if $\mathfrak{a}(d)=\mathfrak{b}(d)$.  
\end{theorem}

\begin{theorem} \label{maximal-eigenspace}
Let $d$ be a positive integer and 
let $\zeta\in \mathbb{C}$ be a primitive $d$-th root of unity.
\begin{enumerate}
\item[(i)] The maximal dimension of the $\zeta$-eigenspace 
of the elements in $G$ 
is $\mathfrak{a}(d)$;
$$\mathrm{max}\, \{\dim V(g,\zeta)\mid g\in G\}=\mathfrak{a}(d).$$
\item[(ii)] Let $x^1,\ldots, x^n$ be a set of basic invariants of $G$. 
Then for any  $g\in G$ with $\dim V(g,\zeta)=\mathfrak{a}(d)$, it holds that 
$x^{\alpha}|_{V(g,\zeta)}=0$ if $d$ does not divide $d_{\alpha}$. 
It also holds that
$\{x^{\alpha}\mid \text{$d$ divides $d_{\alpha}$}\}$
are algebraically independent on $V(g,\zeta)$.
\item[(iii)] If $g\in G$ is a $\zeta$-regular element, then 
$\dim V(g,\zeta)=\mathfrak{a}(d)$.
\end{enumerate}
\end{theorem}

\begin{proof}
For (i) and the second statement of (ii), see \cite[Proposition 11.14]{LehrerTaylor}.
The first statement of (ii) follows from Lemma \ref{dFdz}. For (iii), see 
\cite[Theorem 11.24]{LehrerTaylor}.
\end{proof}
%%%%%%%%%%%%%%%%%%%%%%%%%
\subsection{$(g,\zeta)$-graded coordinates}
\label{sec:graded-coordinates}
%%%%%%%%%%%%%%%%%%%%%%%%%

Let $G$ be a finite complex reflection group acting on $V=\mathbb{C}^n$.
The set of all reflection hyperplanes of $G$ is denoted  $\mathcal{A}$.
We set $$V^{\circ}=V\setminus \bigcup_{H\in \mathcal{A}}H~.$$
For a reflection hyperplane $H\in \mathcal{A}$,
let $e(H)$ be the order of the cyclic group fixing $H$ pointwise
and 
let $L_H\in V^*$ denote a linear map such that $\mathrm{Ker} \,L_H=H$.
For any set of basic invariants $x=(x^1,\ldots, x^n)$
and any linear coordinate system $z=(z^1,\ldots, z^n)$ of $V$,  we have
\begin{equation}\label{detJ}
\det J=\text{(a nonzero constant)}\times
\prod_{H\in \mathcal{A}} L_{H}^{e(H)-1}~,
\end{equation}
where
\begin{equation}\nonumber
J=\begin{bmatrix}
\frac{\partial x^1}{\partial z^1}&\cdots&\frac{\partial x^1}{\partial z^n}
\\
\vdots&\ddots&\vdots
\\
\frac{\partial x^n}{\partial z^1}&\cdots&\frac{\partial x^n}{\partial z^n}
\end{bmatrix}
\end{equation}
is the Jacobian matrix.
See \cite[Theorem 9.8]{LehrerTaylor}.

\begin{remark}\label{xlocal-coord}
Eq.\eqref{detJ}
 implies that the Jacobian matrix $ J$ is invertible on $V^{\circ}$.
Especially, 
$J(q)$ is invertible when $q$ is a regular vector.
Then the inverse function theorem implies  that $x=(x^1,\ldots, x^n)$ makes a 
local coordinate system in a neighborhood of a regular vector $q\in V^{\circ}$.
\end{remark}

\begin{lemma}\label{eigen}
Let $\zeta \in \mathbb{C}$ and let $g\in G$ be a $\zeta$-regular element. 
\begin{enumerate}
\item
$n$ eigenvalues of $g$ are
$\zeta^{1-d_{\alpha}}$ $(1\leq \alpha\leq n)$,
where $d_1,\cdots,d_n$ are the degrees of $G$.
In particular, $\zeta^{d_{\alpha}}=1$ for some $\alpha$.
%(but $\zeta$ is not necessarily a primitive $d_{\alpha}$-th root of unity).
\item
There exists a coordinate system $z=(z^1,\ldots,z^n)$ of $V$ satisfying
$$
g^*z^{\alpha}=\zeta^{d_{\alpha}-1}z^{\alpha}\quad (1\leq \alpha\leq n)~.
$$
\end{enumerate}
\end{lemma}

\begin{proof} 
(1) follows from the fact $J(q)\neq 0$ and Lemma \ref{dFdz}
\cite[Theorem 11.56]{LehrerTaylor}.
\\
(2) By renumbering the basis if necessary,  we have
$$
gq_{\alpha}=\zeta^{1-d_{\alpha}}q_{\alpha}~.
$$
Then the coordinates $z^1,\ldots, z^n$ dual to $q_1,\ldots, q_n$ satisfy
$$
g^*z^{\alpha}=\zeta^{d_{\alpha}-1}z^{\alpha}~.
$$
\end{proof}

\begin{definition}
Let $\zeta \in \mathbb{C}$ and let $g\in G$ be a $\zeta$-regular element. 
We say that  a linear coordinate system $z=(z^1,\ldots, z^n)$  of $V$ is a $(g,\zeta)$-graded coordinate system if 
$$
g^* z^{\alpha}=\zeta^{d_{\alpha}-1} z^{\alpha} \quad (1\leq \alpha\leq n)
$$
holds.
\end{definition}

%%%%%%%%%%%%%%%%%%
\subsection{Admissible triplet}
\label{sec:admissible-triplet}
%%%%%%%%%%%%%%%%%%
In \cite{Satake2020}, Satake defined the notion of admissible triplet.
\begin{definition}
Let $G$ be a finite complex reflection group acting on $V=\mathbb{C}^n$.
Then a triple $(g,\zeta,q)$ consisting of 
\begin{itemize}
\item a primitive $d_1$-th root of unity $\zeta$
(where $d_1$ is the highest degree of $G$),
\item a $\zeta$-regular element $g\in G$,
\item a regular vector $q$ satisfying $gq=\zeta q$,
\end{itemize}
is called {\it an admissible triplet} of $G$.
\end{definition}

\begin{lemma}\label{lem:van1}
Let $(g,\zeta,q)$ be an admissible triplet
and let $z=(z^1,\ldots, z^n)$  be a $(g,\zeta)$-graded system $z$.
\begin{enumerate}
\item If $d_{\alpha}<d_1$, $z^{\alpha}(q)=0$. Moreover,
there exists $1\leq \alpha\leq \mathfrak{a}(d_1)$ such that $z^{\alpha}(q)\neq 0$.
\item
For a $G$-invariant homogeneous polynomial $F\in S[V]^G$ and $a\in \mathbb{Z}_{\geq 0}^n$, the following holds:
$$
\text{if}\quad  \frac{\partial^a F}{\partial  z^a}(q)\neq 0
\quad  \text{then}\quad  \mathrm{deg}F\equiv a\cdot d\mod d_1~.$$
\item \label{van1-2}
 For a  set of basic invariants $x=(x^1,\ldots, x^n)$,
$x^{\alpha}(q)=0$ if $d_{\alpha}<d_1$. 
\item \label{van1-1}
 For a  set of basic invariants $x=(x^1,\ldots, x^n)$ and $a\in \mathbb{Z}^n_{\geq 0}$,
$$
\text{if}\quad
\frac{\partial^a x^{\alpha}}{\partial z^a}(q)\neq 0 ,
\quad \text{then}\quad
d_{\alpha}\equiv \ a\cdot d \mod d_1~.
$$
\end{enumerate}
\end{lemma}
\begin{proof}
When $(g,\zeta,q)$ is an admissible triplet, 
$g$ has eigenvalues $\zeta=\zeta^{1-d_1},\zeta^{1-d_2},\ldots,\zeta^{1-d_n}$
by   Lemma \ref{eigen}
and the regular vector $q$ is an eigenvector of $g$ belonging to the first eigenvalue $\zeta$. 
This proves (1).
(2)  follows from Lemma \ref{dFdz}.
\eqref{van1-2} and \eqref{van1-1} immediately follow from (2).
\end{proof}

To conclude  this section,
we list irreducible finite complex reflection groups which admit 
admissible triplets.
Recall  that an admissible triplet of $G$ exists if and only if 
$\mathfrak{a}(d_1)=\mathfrak{b}(d_1)$ (Theorem \ref{thm:criteria}).

The following irreducible finite complex reflection groups satisfy $\mathfrak{a}(d_1)=\mathfrak{b}(d_1)$ and hence have admissible triplets (see, \cite[Tables D.3, D.5]{LehrerTaylor},\cite[Tables B.1]{OrlikTerao}).
\begin{itemize}
\item  An irreducible finite complex reflection group $G$ is called a {\it duality group} when the degrees and the codegrees of $G$ satisfy the relation
$$d_{\alpha}+d_{\alpha}^*=d_1\quad (1\leq \alpha\leq n). $$
The duality groups include
$$
G(m,m,n),\quad G(m,1,n)
$$
and $26$ groups of exceptional type $$G_{i}\quad 
(4\leq i\leq 37,~~
i\neq 7,11,12,13,15,19,22,31)~.
$$
For a duality group $G$,  $d_1>d_2$ holds. 
Therefore $d_1$ divides $d_{\alpha}$ if and only if  $\alpha=1$
and $d_1$ divides $d_{\alpha}^*$ if and only if $\alpha=1$ 
(i.e. $d_1$ only  divides $d_1^*=0$). 
Thus $\mathfrak{a}(d_1)=\mathfrak{b}(d_1)=1$. 
\item The following groups of rank $2$ satisfy $d_1=d_2$,
$d_1^*=0$, $d_2^*=d_1$. Therefore $\mathfrak{a}(d_1)=\mathfrak{b}(d_1)=2$.
$$
G_{7},~~G_{11},~~G_{19}~.
$$
\item The following groups of rank $2$ satisfy
$d_1>d_2$, $d_1^*=0$, $d_1<d_2^*<2d_2$.
Therefore $\mathfrak{a}(d_1)=\mathfrak{b}(d_1)=1$.
$$
G_{12},~~G_{13},~~G_{22}~.
$$
\item $G_{31}$ has degrees $24,20,12,8$ and codegrees
$0,12,16,28$. Therefore $d_1=24$ and 
$\mathfrak{a}(d_1)=\mathfrak{b}(d_1)=1$.
\end{itemize}

The following groups do not satisfy 
$\mathfrak{a}(d_1)=\mathfrak{b}(d_1)$ and hence they do not have admissible triplets.
\begin{itemize}
\item $G_{15}$ has degrees $24,12$ and codegrees $0,24$.
Therefore $\mathfrak{a}(d_1)=1\neq \mathfrak{b}(d_1)=2$.
\item $G(m,p,n)$ ($1<p<m$, $p|m$) has degrees
$$
m,~~2m,\ldots,~~(n-1)m=d_1,~~\frac{nm}{p}
$$
and codegrees
$$
0,~~m,~~2m,~~\ldots,~~(n-1)m~.
$$
Therefore $\mathfrak{a}(d_1)=1\neq \mathfrak{b}(d_1)=2$.
\end{itemize}

%%%%%%%%%%%%%%%%%%
\section{Good basic invariants}
\label{sec:good}
%%%%%%%%%%%%%%%%%%
In this section, $G$ is a finite complex reflection group acting on $V=\mathbb{C}^n$
with the degrees $d_1\geq d_2\geq \ldots\geq d_n$.
As noted in  Theorem \ref{thm:criteria},
the necessary and sufficient 
condition for the existence of a $d_1$-regular element is 
$\mathfrak{a}(d_1)=\mathfrak{b}(d_1)$.
We assume that $G$ satisfies the condition $\mathfrak{a}(d_1)=\mathfrak{b}(d_1)$.
Since $\mathfrak{a}(d_1)$ is the number of degrees that can be divided by $d_1$, 
we have
$$
d_1=\cdots =d_{\mathfrak{a}(d_1)}>d_{\mathfrak{a}(d_1)+1}\geq \ldots\geq d_n~.
$$
We put 
\begin{equation}\nonumber 
\mathcal{Z}=\{(a_1,\ldots,a_n)\in \mathbb{Z}_{\geq 0}^n \mid 
a_1=\cdots=a_{\mathfrak{a}(d_1)}=0
\}~,
\end{equation}
and
\begin{equation}\label{defIalpha}
\begin{split}
I_{\alpha}^{(k)}&=\{
a\in \mathbb{Z}_{\geq 0}^n\mid 
a\cdot d=d_{\alpha}+kd_1,~~|a|\geq 2
\}
\quad (1\leq \alpha\leq n,~~ k\in \mathbb{Z}_{\geq 0})
~.
\end{split}\end{equation}
We use $I_{\alpha}^{(0)}$ in the definition of good basic invariants  in
\S \ref{subsection-good} and $I_{\alpha}^{(1)}$ in the expression of the potential vector field in Corollary \ref{vector-potential}.
Notice that $a=(a_1,\ldots, a_n)\in I_{\alpha}^{(0)}$ implies that 
$a_{\gamma}=0$ for all $\gamma$ such that $d_{\gamma}\geq d_{\alpha}$.
Especially, $a\in I_{\alpha}^{(0)}$ implies $a\in \mathcal{Z}$.

%%%%%%%%%%%%%%%%%%%%%%%
\subsection{Compatible basic invariants}
\label{subsection-compatible}
%%%%%%%%%%%%%%%%%%%%%%%
\begin{definition}
Let $(g,\zeta,q)$ be an admissible triplet of $G$
and let $z=(z^1,\ldots, z^n)$  be a $(g,\zeta)$-graded coordinate system.
We say that a set of basic invariants $x=(x^1,\ldots, x^n)$ is {\it compatible} 
at $q$ with $z$ if
\begin{equation}\label{compatible-x0}
\frac{\partial x^{\alpha}}{\partial z^{\beta}}(q)=\delta^{\alpha}_{\beta}
\quad (1\leq \alpha,\beta\leq n)~.
\end{equation}
\end{definition}

\begin{proposition}\label{uniqueness-compatible}
For any admissible triplet $(g,\zeta,q)$ and any  $(g,\zeta)$-graded coordinate system $z$,
there exists a set of basic invariants  which  is compatible with $z$  at $q$.
\end{proposition}

\begin{proof} 
Take any set of basic invariants $x=(x^1,\ldots, x^n)$.
We construct from $x$ another set of basic invariants which is compatible with $z$ at $q$.
Applying Lemma \ref{lem:van1} \eqref{van1-1} to the case $a=\bm{e}_{\beta}$,  we see that 
\begin{equation}\nonumber
\frac{\partial x^{\alpha}}{\partial z^{\beta}}(q)=0 
\end{equation}
holds if  
$d_{\alpha}\neq  d_{\beta}$.
This means that the Jacobian matrix $J(q)=(\frac{\partial x^{\alpha}}{\partial z^{\beta}}(q))$, evaluated at the regular vector $q$,
is block diagonal
with
each block consisting of $\alpha$'s with the same degree $d_{\alpha}$. 
So $J(q)$ is written as follows.
$$
J(q)=\begin{bmatrix}
\frac{\partial x^1}{\partial z_1}(q)&\cdots&\frac{\partial x^1}{\partial z^n}(q)
\\
\vdots&\ddots&\vdots
\\
\frac{\partial x^n}{\partial z_1}(q)&\cdots&\frac{\partial x^n}{\partial z^n}(q)
\end{bmatrix}=
\begin{bmatrix}
J_1&&O\\
&\ddots&\\
O&&J_k
\end{bmatrix}~.
$$
Since $q$ is a regular vector,  $J(q)$ is invertible by Remark \ref{xlocal-coord}.
Hence each block $J_i$ is invertible.
If we put
\begin{equation}\nonumber
\begin{bmatrix}\tilde{x}^1\\\vdots\\ \tilde{x}^n\end{bmatrix}
=\begin{bmatrix}
J_1^{-1}&&O\\&\ddots&\\O&&J_k^{-1}
\end{bmatrix}
\begin{bmatrix}
x^1\\\vdots\\x^n
\end{bmatrix}~,
\end{equation}
we obtain a new set of basic invariants $\tilde{x}=(\tilde{x}^1,\ldots,\tilde{x}^n)$ with $\mathrm{deg}\,\tilde{x}^{\alpha}=d_{\alpha}$. Then we have
$$
\begin{bmatrix}
\frac{\partial \tilde{x}^1}{\partial z_1}(q)&\cdots&
\frac{\partial \tilde{x}^1}{\partial z^n}(q)
\\
\vdots&\ddots&\vdots
\\
\frac{\partial \tilde{x}^n}{\partial z_1}(q)&\cdots&
\frac{\partial \tilde{x}^n}{\partial z^n}(q)
\end{bmatrix}=
\begin{bmatrix}
\frac{\partial \tilde{x}^1}{\partial x_1}(q)&\cdots&
\frac{\partial \tilde{x}^1}{\partial x^n}(q)
\\
\vdots&\ddots&\vdots
\\
\frac{\partial \tilde{x}^n}{\partial z_1}(q)&\cdots&
\frac{\partial \tilde{x}^n}{\partial z^n}(q)
\end{bmatrix}
\begin{bmatrix}
\frac{\partial {x}^1}{\partial z_1}(q)&\cdots&
\frac{\partial {x}^1}{\partial z^n}(q)
\\
\vdots&\ddots&\vdots
\\
\frac{\partial {x}^n}{\partial z_1}(q)&\cdots&
\frac{\partial {x}^n}{\partial z^n}(q)
\end{bmatrix}
=I.
$$
Thus $\tilde{x}$ is compatible with $z$ at $q$.
\end{proof}

The next lemma will be useful in the proof of 
Theorem \ref{existence-good}.
\begin{lemma}\label{a-b}
Let $(g,\zeta,q)$ be an admissible triplet of $G$,
$z=(z^1,\ldots, z^n)$ a $(g,\zeta)$-graded coordinate system
and 
let $x$ be a set of basic invariants
compatible at $q$ with  $z$.
\begin{enumerate}
\item 
For $a\in \mathbb{Z}_{\geq 0}^n$ and $b \in \mathcal{Z}$, the following holds.
$$
\text{If  }  \frac{\partial^a x^b}{\partial z^a}(q)\neq 0,  \text{  then }
\begin{cases}
\text{either  } (|a|>|b| \text{ and } a\cdot d=b\cdot d+k d_1 (k\in \mathbb{Z}_{\geq 0})), \\
\text{or } a=b
\end{cases}~.
$$
\item 
For $a\in \mathcal{Z}$, 
\begin{equation}\nonumber
\frac{\partial^a x^a}{\partial z^a}(q)= a! ~.
\end{equation}
\end{enumerate}
\end{lemma}
\begin{proof} 
(1) If $b=0$, i.e. if $x^b$ is a constant, $\frac{\partial^a x^b}{\partial z^a}(q)\neq 0$
implies that $a=0$, and the claim holds.
For $b\neq 0$, $x^b$ is written in the following form:
$$
x^b=x^{\beta_1}x^{\beta_2}\cdots x^{\beta_{|b|}}.
$$
Here $d_{\beta_i}<d_1$ holds because we assumed $b\in \mathcal{Z}$.
Differentiating $x^b$ by $z^a$, we  have
\begin{equation}\label{eq:a-b}
\frac{\partial^a x^b}{\partial z^a}(q)=
\sum_{\begin{subarray}{c}
a(1),\ldots, a(|b|)\in \mathbb{Z}_{\geq 0}^n;\\
\\a=a(1)+\cdots+a(|b|)\end{subarray}}
\frac{a!}{a(1)!\cdots a(|b|)!}
\underbrace{\prod_{i=1}^{|b|} 
\frac{\partial^{a(i)}x^{\beta_i}}{\partial z^{a(i)}}(q)}_{(\star)}~.
\end{equation}
If $\frac{\partial x^b}{\partial z^a}(q)\neq 0$, then $(\star)\neq 0$ for at least one
tuple $(a(1),\ldots, a(|b|))$.
Then, by Lemma \ref{lem:van1}, 
\begin{equation}\label{temp}
a(i)\cdot d= d_{\beta_i}+k_id_1\quad  (k_i\in \mathbb{Z}) 
\end{equation}
holds for all $1\leq i\leq |b|$.
Given that $0<d_{\beta_i}<d_1$ and that $a(i)\cdot d\geq 0$,
$k_i$ cannot be negative. Moreover, $a(i)=0$ does not satisfy \eqref{temp}
since $d_{\beta_i}\not\equiv 0 \mod d_1$.
So $k_i\geq 0$ and $|a(i)|\geq 1$.
When $|a(i)|=1$, $\frac{\partial^{a(i)}x^{\beta_i}}{\partial z^{a(i)}}(q)\neq 0$
if  and only  if $\frac{\partial}{\partial z^{a(i)}}=\frac{\partial}{\partial  z^{\beta_i}}$ by the compatibility condition \eqref{compatible-x0}.
Therefore  if $(\star)\neq 0$,  we have
\begin{equation}\label{i-condition}
( |a(i)|\geq 2~~ \text{and} ~~
a(i)\cdot d=d_{\beta_i} + k_id_1~~(k_i\in \mathbb{Z}_{\geq 0}))\quad \text{or}\quad (|a(i)|=1 ~~\text{ and } ~~a(i)=\bm{e}_{\beta_i})
\end{equation}
for all $1\leq i\leq |b|$. 
Summing over $1\leq i\leq |b|$, we obtain
$$
|a|>|b|~~\text{and}~~ a\cdot d=b\cdot d+kd_1 ~(k\in \mathbb{Z}_{\geq 0})
\quad \text{or}\quad (|a|=|b|~~\text{and}~~a=b)~.
$$
This prove  the statement (1).
\\
(2) When $a=0$, the statement is clear.
When $a\neq 0$, 
 non-vanishing terms  in \eqref{eq:a-b} with $a=b\in \mathcal{Z}$ satisfy $a(i)=\bm{e}_{\beta_i}$
 for all $1\leq i\leq |b|$.
 There are $a_1! \cdots a_n!=a!$ such terms, and we have 
$$\frac{\partial x^a}{\partial z^a}(q)=
a!\prod_{\alpha=1}^n \left(\frac{\partial x^{\alpha}}{\partial z^{\alpha}}(q)
\right)^{a_{\alpha}}
=a!~
$$
by the compatibility condition \eqref{compatible-x0}.
\end{proof}

We will need the following lemma in \S \ref{sec:duality-groups}.
There we assume $\mathfrak{a}(d_1)=\mathfrak{b}(d_1)=1$
and in that case,  the assumption $z^{2}(q)=\cdots=z^n(q)=0$ holds true.
\begin{lemma}\label{compatible-x-form}
Let $(g,\zeta,q)$ be an admissible triplet of $G$
and let $z=(z^1,\ldots, z^n)$  be a $(g,\zeta)$-graded coordinate system
such that $z^{2}(q)=\cdots=z^n(q)=0$.
Then a set of basic invariants
$x=(x^1,\ldots, x^n)$ which is compatible at $q$ with $z$
takes the following form
\begin{equation}\nonumber
\begin{split}
x^1&=\frac{1}{d_1 (z^1(q))^{d_1-1}}(z^1)^{d_1}+
f_1~,
\\
x^{\alpha}&=
\frac{1}{ (z^1(q))^{d_{\alpha}-1}}
(z^1)^{d_{\alpha}-1} z^{\alpha}+f_{\alpha}
\quad(\alpha\neq 1),
\end{split}
\end{equation}
where 
$f_{\alpha} \in \mathbb{C}[z]$ is some polynomial whose degree in $z^1$ is 
less than $d_{\alpha}-1$.
As a consequence  we have
$$
x^{\alpha}(q)=\begin{cases}\frac{z^1(q)}{d_1}&(\alpha=1)\\0&(\alpha\neq 1)
\end{cases}~.
$$
\end{lemma}

\begin{proof}
Let us write 
$$
x^{\alpha}=A^{\alpha} (z^1)^{d_{\alpha}}
+\sum_{\beta=2}^n B^{\alpha}_{\beta}(z^1)^{d_{\alpha}-1} z^{\beta}+f_{\alpha}
$$
where 
$f_{\alpha} \in \mathbb{C}[z]$ is some polynomial whose degree in $z^1$ is 
less than $d_{\alpha}-1$.

For $\alpha=1$,
 the compatibility condition means 
 $$
 \frac{\partial x^1}{\partial z^1}(q)=d_1A^1( z^1(q))^ {d_1-1}=1,\quad 
 \frac{\partial x^1}{\partial z^{\gamma}}(q)
 =B_{\gamma}^1( z^1(q))^ {d_1-1}=0.
 $$
Therefore $A^1=\frac{1}{d_1(z^1(q))^{d_1-1}}$ and $B_{\gamma}^1=0$.
 For $\alpha\neq 1$, 
 the compatibility condition means 
 $$
 \frac{\partial x^{\alpha}}{\partial z^1}(q)
 =d_{\alpha}A^{\alpha}( z^1(q))^ {d_{\alpha}-1}=0,\quad 
 \frac{\partial x^{\alpha}}{\partial z^{\gamma}}(q)
 =B_{\gamma}^{\alpha}( z^1(q))^ {d_{\alpha}-1}=\delta^{\alpha}_{\gamma}
 ~~(\gamma\neq 1).
 $$
 Therefore $A^{\alpha}=0$, $B_{\alpha}^{\alpha}=\frac{1}{(z^1(q))^{d_{\alpha}-1}}$
 and $B_{\gamma}^{\alpha}=0$ ($\gamma\neq \alpha$).
\end{proof}

%%%%%%%%%%%%%%%%%%%%
\subsection{Good basic invariants}
\label{subsection-good}
%%%%%%%%%%%%%%%%%%%%
In \cite[Definition 3.3]{Satake2020}, Satake defined a set of good basic invariants
using a graded ring isomorphism. 
His elegant definition makes it clear the existence and the uniqueness.
However in this article, we adopt another equivalent definition 
which Satake gave in his talk \cite{Satake2019}.
See \cite[Proposition 5.2]{Satake2020} for a proof of the equivalence.

\begin{definition}\label{def-good} 
A set of basic invariants $x=(x^1,\ldots, x^n)$
is {\it good} with respect to an admissible triplet $(g,\zeta,q)$
if
\begin{equation}\label{good0}
\frac{\partial^{a} x^{\alpha}}{\partial z^{a}}(q)
=0 \quad (1\leq \alpha\leq n,~~ a\in I_{\alpha}^{(0)}),
\end{equation}
where $z=(z^1,\ldots, z^n)$ is a $(g,\zeta)$-graded coordinate system of $V$.
\end{definition}

\begin{lemma}
The definition of  good basic invariants does not depend on the choice of
$(g,\zeta)$-graded coordinate system $z$.
\end{lemma}
\begin{proof}
Let $z=(z^1,\ldots, z^{n})$ and $\tilde{z}=(\tilde{z}^1,\ldots, \tilde{z}^n)$ be 
$(g,\zeta)$-graded coordinate systems.
Since these coordinate systems are associated to the  eigenspace decomposition
of $g$, 
each ${z}^{\alpha}$ is a linear combination of $\tilde{z}^{\beta}$'s  with $d_{\beta}=d_{\alpha}$.
The chain rule implies that
$$
\frac{\partial^b}{\partial {\tilde{z}}^{b}}=
\sum_{\begin{subarray}{c}
a\in \mathbb{Z}_{\geq 0}^n
\\
|a|=|b|,a\cdot d=b\cdot d\end{subarray}} B_{b,a}\frac{\partial^a }{\partial {z}^a} 
\quad (b\in \mathbb{Z}_{\geq 0}^n)
$$
holds with some constants $B_{b,a}\in \mathbb{C}$.
If $b\in  I_{\alpha}^{(0)}$, $a$'s appearing  in the sum also  belong  to $I_{\alpha}^{(0)}$.
Therefore, if \eqref{good0} holds, $\frac{\partial^b x^{\alpha}}{\partial \tilde{z}^b}(q)=0$
holds for $b\in I_{\alpha}^{(0)}$.
Thus if $x$
is good with respect to $z$,
then $x$ is also good with respect to $\tilde{z}$ and vice versa.
\end{proof}

\begin{theorem} \label{existence-good} 
A set of good basic invariants with respect to a given admissible triplet  exists.
\end{theorem}

\begin{proof}
Let $(g,\zeta,q)$ be an admissible triplet and 
let $z=(z^1,\ldots, z^n)$ be a $(g,\zeta)$-graded coordinate system. 
Take a set of (not necessarily good) basic invariants
$x=(x^1,\ldots, x^n)$ which is compatible with  $z$ at $q$.
Let us consider a set of basic invariants 
$y=(y^1,\ldots,y^n)$ written in the following form:
\begin{equation}\nonumber
y^{\alpha}=x^{\alpha}+\sum_{b\in I_{\alpha}^{(0)}}
B_{b}^{\alpha} x^{b} 
\quad (B_{b}^{\alpha}\in \mathbb{C})~.
\end{equation}
We show that there exists a set of coefficients $\{B_b^{\alpha}\mid a\in I_{\alpha}^{(0)}\}$ such that
$y$ is good.

Assume that $y^{\alpha}$ satisfies the goodness condition \eqref{good0}.
We differentiate the both sides by $z^a$ with $a\in I_{\alpha}^{(0)}$.
Then by Lemma \ref{a-b}, we have
\begin{equation}\nonumber
\begin{split}
0=\frac{\partial^a y^{\alpha}}{\partial z^a}(q)&=
\frac{\partial^a x^{\alpha}}{\partial z^a}(q)+
a! B_a^{\alpha} 
+\sum_{\begin{subarray}{c}b\in I_{\alpha}^{(0)};
\clubsuit
\end{subarray}}
B_b^{\alpha}\frac{\partial x^b}{\partial z^a}(q)~.
\end{split}
\end{equation}
Here $\clubsuit$ is the condition
$$
(\clubsuit)\quad 
|b|<|a| \text{ and }
b\cdot d\equiv a\cdot d\mod d_1 \text{ and } b\cdot d\leq a\cdot d~.
$$
When $|a|=2$,  $b\in I_{\alpha}^{(0)}$ satisfying $\clubsuit$ does not exist 
since $b\in I_{\alpha}^{(0)}$ implies $|b|\geq 2$.
So  we obtain
$$
B_a^{\alpha}=-\frac{1}{a!}\frac{\partial^a x^{\alpha}}{\partial z^a}(q)~.
$$
When $|a|>2$, we have
\begin{equation}\nonumber
B_a^{\alpha} 
=-\frac{1}{a!}\left(\frac{\partial^a x^{\alpha}}{\partial z^a}(q)
+\sum_{\begin{subarray}{c}b\in I_{\alpha}^{(0)};\clubsuit\end{subarray}}
B_b^{\alpha}\frac{\partial x^b}{\partial z^a}(q)
\right)~.
\end{equation}
Thus we can solve this system of linear equations
from smaller $|a|$ to  larger $|a|$ and obtain 
a unique solution $\{B_a^{\alpha}\mid a\in I_{\alpha}^{(0)}\}$.
\end{proof}

To show the uniqueness, we need the following improved version of Lemma \ref{a-b}. 
While $k\geq 0$ in 
Lemma \ref{a-b},  we have $k\geq 1$ in the next lemma
due to the goodness assumption on $x$.
\begin{lemma}\label{a-b-improved}
Let $x$ be a set of basic invariants
which is good with respect to an admissible triplet $(g,\zeta,q)$
and is 
compatible at $q$ with a $(g,\zeta)$-graded coordinate system $z$ at $q$.
For $a\in \mathbb{Z}_{\geq 0}^n$ and $b \in \mathcal{Z}$, the following holds.
$$
\text{If }~~\frac{\partial^a x^{b}}{\partial z^a}(q)\neq 0, ~~{ then }~~
\begin{cases} \text{either}~~
(|a|>|b|\text{ and }a\cdot d=b\cdot d+kd_1 ~~(k\in\mathbb{Z}_{\geq 1}))
\\ \text{ or }~~a=b
\end{cases}~.
$$ 
\end{lemma}
\begin{proof}
The proof is almost the same as that of Lemma \ref{a-b}.
Taking the goodness of $x$ into consideration,  
the condition for $k_i$ in 
\eqref{i-condition} becomes  $k_i\in \mathbb{Z}_{\geq 1}$.
Summing over $1\leq i\leq |b|$, we obtain the statement.
\end{proof}

\begin{theorem} \label{uniqueness-good} 
 If  both $x=(x^1,\ldots,x^n)$ and $y=(y^1,\ldots,y^n)$ are sets of good basic invariants
with respect to the same admissible triplet,
then
$$
y^{\alpha}=\sum_{\begin{subarray}{r}1\leq \beta\leq n;\\
 d_{\beta}=d_{\alpha}\end{subarray}} 
 A^{\alpha}_{\beta} x^{\beta} \quad 
(1\leq \alpha\leq n)
$$
where $A^{\alpha}_{\beta}$ are some constants.
In other words,
if $x$  and $y$ are sets of good basic invariants
with respect to the same admissible triplet,
they span the same vector subspace of $S[V]^G$.
\end{theorem}

\begin{proof}
Let $(g,\zeta,q)$ be an admissible triplet
and let $x=(x^1,\ldots,x^n)$  be a set  of good basic invariants.
Take a  $(g,\zeta)$-graded coordinate system $z=(z^1,\ldots, z^n)$
and construct from $x$ a  set of  basic invariants 
$\tilde{x}=(\tilde{x}^1,\ldots,\tilde{x}^n)$  which is compatible with
$z$ at $q$, 
by the method of
the  proof of Proposition \ref{uniqueness-compatible}.
Then clearly $\tilde{x}^{\alpha}$'s are linear combinations of $x^{\beta}$'s with
$d_{\beta}=d_{\alpha}$.
Therefore as $x$ is good, so is $\tilde{x}$.

Let $y=(y^1,\ldots, y^n)$ be another set of good basic invariants.
By the degree consideration, each $y^{\alpha}$  is expressed as
$$
y^{\alpha}=\sum_{\begin{subarray}{r}1\leq \beta\leq n;\\
 d_{\beta}=d_{\alpha}\end{subarray}} 
 A^{\alpha}_{\beta}\tilde{x}^{\beta}+
\sum_{b\in I_{\alpha}^{(0)}} B^{\alpha}_b \tilde{x}^b~.
$$
Let us differentiate the equation by $z^a$ $(a\in I_{\alpha}^{(0)})$
and evaluate at $q$. 
The LHS 
and the first term in the RHS vanish because $y$ and $\tilde{x}$ are good .
The second term in the RHS 
 does not vanish only when $a=b$ by Lemma \ref{a-b-improved}.
Therefore we have
$$
0=B_a^{\alpha}\frac{\partial \tilde{x}^a}{\partial z^a}(q)=
a! B_a^{\alpha}$$
by Lemma \ref{a-b}.
 So we have  $B_a^{\alpha}=0 $ for all $a\in I_{\alpha}^{(0)}$ and 
 $$
 y^{\alpha}=
 \sum_{\begin{subarray}{r}1\leq \beta\leq n;\\
 d_{\beta}=d_{\alpha}\end{subarray}}
  A^{\alpha}_{\beta}\tilde{x}^{\beta}~.
 $$
 Since $\tilde{x}^{\beta}$'s  are  linear combinations of $x^{\gamma}$'s
 with the same degree, so are
 $y^{\alpha}$'s. 
 \end{proof}
 
\begin{remark}
Theorem \ref{uniqueness-good}  implies that
the goodness condition determines the subspace $\mathcal{I}$ of $S[V]^G$ spanned by good basic invariants. On the other hand, imposing the compatibility with a $(g,\zeta)$-graded coordinates on good basic invariants 
corresponds to choosing a basis of $\mathcal{I}$.
\end{remark}

The  next   lemma  is an improved version  of 
Lemma  \ref{lem:van1} \eqref{van1-1}.   We  will  use the lemma in \S \ref{sec:duality-groups}.
\begin{lemma}\label{x-derivatives}
Let $x$ be a set of basic invariants
which is good with respect to an admissible triplet $(g,\zeta,q)$
and which is compatible with  a $(g,\zeta)$-graded coordinate system $z$ at $q$.
For $a\in \mathbb{Z}_{\geq 0}^n$ with $|a|\geq 2$ and 
$1\leq\beta\leq n$, the  following holds.
$$
\text{If }~~\frac{\partial^a x^{\beta}}{\partial z^a}(q)\neq 0, ~~{ then }~~
a\cdot d=d_{\beta}+kd_1 ~~(k\in\mathbb{Z}_{\geq 1})~.
$$
\end{lemma}
\begin{proof}
By Lemma \ref{lem:van1} \eqref{van1-1}, 
$\displaystyle{
\frac{\partial^a x^{\beta}}{\partial z^a}(q)\neq 0
}$
implies   $a\cdot d=d_{\beta}+kd_1$ $(k\in \mathbb{Z})$.  The integer $k$ cannot be  negative since otherwise  $a\cdot  d\leq 0$, contradicting the assumption $|a|\geq 2$. 
Moreover $k$ cannot be  zero since otherwise $a$ belongs to $I_{\beta}^{(0)}$,  contradicting to
the assumption that $x$ is good.
\end{proof}
%%%%%%%%%%%%%%%%%%%%%%%%%%%%%
\subsection{Independence of good basic invariants 
of the choice of admissible triplet 
when $\mathfrak{a}(d_1)=\mathfrak{b}(d_1)=1$ }
%%%%%%%%%%%%%%%%%%%%%%%%%%%%%
In this subsection, we assume that the finite complex reflection group $G$ satisfies $\mathfrak{a}(d_1)=\mathfrak{b}(d_1)=1$ (i.e. $d_1>d_2$).
The assumption implies that the $\zeta$-eigenspace of $g$ is one-dimensional
for any admissible triplet $(g,\zeta,q)$.
(see Theorem \ref{maximal-eigenspace} (iii)).

\begin{proposition}\label{indep}
Let $G$ be a  finite complex reflection group with $\mathfrak{a}(d_1)=\mathfrak{b}(d_1)=1$.
Then, 
if a set of basic invariants $x$ is good with respect to 
an admissible triplet of $G$,  then  $x$ is good with respect to
any admissible triplet of $G$.
\end{proposition}

\begin{proof}
Let $(g,\zeta,q)$ and $(\tilde{g},\tilde{\zeta},\tilde{q})$ be admissible triplets of $G$.
Since both $\zeta$ and $\tilde{\zeta}$ are primitive $d_1$-th roots of unity,
$\tilde{\zeta}=\zeta^j$ for some integer $j$.
Then $g^j$ is $\tilde{\zeta}$-regular with $g^j q=\tilde{\zeta}q$.
Since all $\tilde{\zeta}$-regular elements are conjugate in $G$ (Remark \ref{remark:conjugacy}), 
there exists $h\in G$ such that
$\tilde{g}=hg^j h^{-1}$. Then we have
$$
\tilde{g}(hq)=h g^j q=\tilde{\zeta} (hq)~.
$$
Given that $\mathfrak{a}(d_1)=1$, the $\tilde{\zeta}$-eigenspace of $\tilde{g}$ is one-dimensional.
So $hq$ is a nonzero scalar multiple of  $\tilde{q}$.
Therefore the goodness with respect to $(\tilde{g},\tilde{\zeta},\tilde{q})$
is equivalent to   that with respect to $(\tilde{g},\tilde{\zeta},hq)$.

Now consider the isomorphism $h:V\to V$ which sends 
$v\in V$ to $hv$.
It we
take a $(g,\zeta)$-graded coordinate system $z=(z^1,\ldots,z^n)$ of $V$,
then  $\tilde{z}^{\alpha}:=h^*z^{\alpha}$ 
is a $(\tilde{g},\tilde{\zeta})$-graded coordinate system
since
$$
\tilde{g}^*\tilde{z}^{\alpha}=\tilde{g}^*h^*z^{\alpha}
=h^*(g^j)^*z^{\alpha}=h^* (\zeta^{d_{\alpha}-1})^j z^{\alpha}
=\tilde{\zeta}^{d_{\alpha}-1} \tilde{z}^{\alpha}~.
$$
Therefore
$$
\frac{\partial {x}^{\alpha}}{\partial {z}^a}({q})=0~~(a\in I_{\alpha}^{(0)})
\Longleftrightarrow
\frac{\partial (h^*{x}^{\alpha})}{\partial \tilde{z}^a}(hq)=0~~(a\in I_{\alpha}^{(0)})~.
$$
Notice that $h^*x^{\alpha}=x^{\alpha}$ holds since $x^{\alpha}$ is $G$-invariant.
Thus this implies that if $x$ is good with respect to $(g,\zeta,q)$,
$h^*x=x$  is good with respect to ($\tilde{g},\tilde{\zeta},hq)$.
\end{proof}

Theorem \ref{uniqueness-good}  and Proposition \ref{indep} imply the following
\begin{theorem}\label{unique}
For a finite complex reflection group $G$ with $\mathfrak{a}(d_1)=\mathfrak{b}(d_1)=1$,
the vector subspace $\mathcal{I}$ in $S[V]^G$ spanned by a set of good basic invariants
is uniquely determined by $G$.
\end{theorem}

\begin{remark}\label{ad=2}
As for the case
 $\mathfrak{a}(d_1)=\mathfrak{b}(d_1)>1$, the same statement holds if
 $G$ is irreducible.
Indeed, irreducible finite complex reflection group with
 $\mathfrak{a}(d_1)=\mathfrak{b}(d_1)>1$
 are $G_7,G_{11},G_{19}$ (see \S \ref{sec:admissible-triplet}).
These groups has rank two and $d_1=d_2$.
Therefore the goodness condition \eqref{good0} is empty and
any set of basic invariants is good with respect to any admissible triplet.
 Thus the vector subspace $\mathcal{I}$ in $S[V]^G$ spanned by a set of good basic invariants is nothing but the degree $d_1(=d_2)$ part of $S[V]^G$.
For these groups any $d_1$-regular element $g$ 
is equal to $\zeta I$ with some primitive $d_1$-th root of unity $\zeta$,
and $(g,\zeta)$ together with any regular vector $q$ form an admissible triplet.
\end{remark}

%%%%%%%%%%%%%%%%%%%%%%%%
\section{Introduction for the second part}
\label{second-part}
%%%%%%%%%%%%%%%%%%%%%%%%
In the remaining part of this article, we study the relationship between
the good basic invariants and the natural Saito structure for duality groups 
studied in \cite{KatoManoSekiguchi2015} \cite{Arsie-Lorenzoni2016} \cite{KMS2018}.
In this section, we recall necessary definitions and properties  from \cite{KMS2018}.
To prevent  the article becoming too  lengthy,  we keep the explanation to a bare minimum.
For the  definition of Saito structure and  the construction of
natural Saito structure, see loc.cit.

Let $G$ be a finite complex reflection group acting on $V=\mathbb{C}^n$.
Let  $z$ be a linear coordinate system of $V$ and let
$x$ be a set of basic invariants.

As in \S \ref{sec:graded-coordinates},  we put
\begin{equation}\label{defJ}
J^{\gamma}_{\beta}=\frac{\partial x^{\gamma}}{\partial z^{\beta}}~,\quad
\tilde{J}^{\gamma}_{ \beta}=\frac{\partial z^{\gamma}}{\partial x^{\beta}}\quad
(1\leq \beta,\gamma\leq n)~.
\end{equation}
$J^{\gamma}_{\beta}$ is a homogeneous polynomial in $z$ while
$\tilde{J}^{\gamma}_{\beta}$ is a homogeneous rational function  in $z$ having poles at reflection hyperplanes.
 The $n\times n$ matrix 
whose $(\gamma,\beta)$-entries are $J^{\gamma}_{\beta}$ (resp. $\tilde{J}^{\gamma}_{\beta}$) is denoted $J$ (resp. $\tilde{J}$).
$J$ is the Jacobian matrix of $x$ and 
$\tilde{J}$ is  the inverse matrix  of $J$.
We also put
\begin{equation}\nonumber 
\Omega_{\alpha\beta}^{\gamma}=-\sum_{\mu,\nu=1}^n
\frac{\partial z^{\mu}}{\partial x^{\alpha}}
\frac{\partial z^{\nu}}{\partial x^{\beta}}
\frac{\partial^2 x^{\gamma}}{\partial z^{\mu}\partial z^{\nu}}
\quad (1\leq \alpha,\beta,\gamma\leq n)~.
\end{equation}
Each $\Omega_{\alpha\beta}^{\gamma}$ is  a homogeneous $G$-invariant 
rational function having poles at reflection hyperplanes.
The $n\times n$ matrix whose $(\gamma,\beta)$-entries are $\Omega_{\alpha\beta}^{\gamma}$
is denoted $\Omega_{\alpha}$.
$\Omega_{\alpha}$ is the local connection form 
of the trivial connection on $TV^{\circ}$, with respect  to the local frame
$\{\frac{\partial}{\partial x^{\alpha}}\}$,
where $V^{\circ}$ is the complement of all reflection hyperplanes in $V$.
Given that 
$$
\sum_{\alpha=1}^n z^{\alpha}\frac{\partial }{\partial z^{\alpha}}
=\sum_{\alpha=1}^n d_{\alpha}x^{\alpha}\frac{\partial }{\partial x^{\alpha}}
$$
acts as  a degree operator on $\mathbb{C}[z]$,  we have
\begin{equation}\label{Omega}
\sum_{\alpha=1}^n d_{\alpha}x^{\alpha}\Omega_{\alpha}
=\mathrm{diag}(1-d_1,\ldots,1-d_n)~.
\end{equation}
The discriminant polynomial  of $G$ is a $G$-invariant homogeneous polynomial  defined by
\begin{equation}\label{defDelta}
\Delta=\prod_{H\in \mathcal{A}} L_H^{e_H} \in  \mathbb{C}[x]~.
\end{equation}
The degree of $\Delta$ is $\sum_{\alpha=1}^n (d_{\alpha}+d_{\alpha}^*)$
\cite[Theorem 6.42, Corollary 6.63]{OrlikTerao}.
In \cite[Proposition 6.1]{KMS2018}, it is shown that
$$\Delta \Omega_{\alpha\beta}^{\gamma} \in \mathbb{C}[x]~,
\quad
\mathrm{deg}\, \Delta\Omega_{\alpha\beta}^{\gamma}=\mathrm{deg}\Delta+d_{\gamma}-d_{\alpha}-d_{\beta}~.
$$

If $G$ is a duality group, i.e.  if $G$ is 
an irreducible finite complex reflection group  satisfying the relation 
$d_{\alpha}+d_{\alpha}^*=d_1$ $(1\leq \alpha\leq n)$, 
then $\mathfrak{a}(d_1)=\mathfrak{b}(d_1)=1$, and hence $d_1>d_2$.
We write $\mathbb{C}[x']$ for $\mathbb{C}[x^2,\ldots, x^n]$.
It is known that
$\mathrm{deg}\Delta=nd_1$ and 
the coefficient of $(x^1)^{n}$ in $\Delta$ is 
a non-zero constant \cite{OrlikSolomon}\cite{Bessis2015}.
Multiplying by a nonzero constant if necessary,
we assume that the coefficient is one.
Since $\mathrm{deg}\,\Delta \Omega_{\alpha}<(n+1)d_1$, its degree in $x^1$ is at most $n$.
We write  
$\Gamma_{\alpha\beta}^{\gamma}\in \mathbb{C}[x']$ for
the coefficient of $(x^1)^n$
and $D_{\alpha\beta}^{\gamma}\in \mathbb{C}[x']$ for
the coefficient of $(x^1)^{n-1}$:
\begin{equation}\nonumber 
\begin{split}
\Delta \Omega_{\alpha\beta}^{\gamma}&
=(x^1)^n\Gamma_{\alpha\beta}^{\gamma} +(x^1)^{n-1} D_{\alpha\beta}^{\gamma}
+\cdots~.
\end{split}\end{equation}
Then
$\mathrm{deg}\,\Gamma_{\alpha\beta}^{\gamma}=d_{\gamma}-d_{\alpha}-d_{\beta}$
and  we have
\begin{equation}\label{D1-Gamma}
D_{1\beta}^{\gamma}=\frac{1-d_{\gamma}}{d_1}\delta^{\gamma}_{\beta}
-\sum_{\alpha=2}^n \frac{d_{\alpha}}{d_1}x^{\alpha}\Gamma_{\alpha\beta}^{\gamma}
\end{equation}
by \eqref{Omega}.
For a duality group, the matrix $\Omega_1$ is invertible and its inverse $\Omega_1^{-1}=\tilde{\Omega}^{\gamma}_{\beta}$  satisfies
\begin{equation}\label{invOmega}
\tilde{\Omega}-D_1^{-1}x^1\in M(n,\mathbb{C}[x']), \quad\deg \tilde{\Omega}^{\gamma}_{\beta}=
d_1+d_{\gamma}-d_{\beta}~.
\end{equation}
See \cite[Proposition 6.1, Lemma 7.2]{KMS2018}.
Moreover
\begin{equation}\label{def-C}
C_{\alpha}:=\Omega_1^{-1}(\Omega_{\alpha}-\Gamma_{\alpha})\in M(n,\mathbb{C}[x']),
\quad
\mathrm{deg}\,C_{\alpha\beta}^{\gamma}=d_1+d_{\gamma}-d_{\alpha}-d_{\beta}~.
\end{equation}
When we take an admissible triplet $(g,\zeta,q)$,
we can take a $(g,\zeta)$-graded coordinate system as $z$.
We can take as $x$
a set of good basic invariants which is compatible with $z$ at $q$. 
Then we have $\Gamma_{\alpha\beta}^{\gamma}=0$
(Theorem \ref{good-implies-flat}).
Moreover  coefficients of $C_{\alpha\beta}^{\gamma}$ are expressed as
derivatives of $x^{\gamma}$  by $z$.
(Theorem \ref{main-theorem2}).
This is a generalization of Satake's result \cite{Satake2020}.
From the viewpoint of the natural Saito structure studied in \cite{KMS2018},
$\Gamma_{\alpha}$ and 
$C_{\alpha}$ are matrix expressions of its affine connection and multiplication
with respect to the  frame $\{\frac{\partial}{\partial x^{\alpha}}\}$.
Especially, that $\Gamma_{\alpha\beta}^{\gamma}=0$ implies that
a set of good basic invariants is a set of flat invariants.

If $G$ is an irreducible finite Coxeter group,  (then $G$ is a duality group), 
the degrees of $G$ satisfy the following relations:
\begin{equation}\nonumber 
d_1>d_2,\quad d_n=2,\quad 
d_{\alpha}+d_{n+1-\alpha}=d_1+2\quad (1\leq \alpha\leq n)~.
\end{equation}
Moreover,  there exists a 
a nondegenerate $G$-invariant symmetric bilinear form
$\langle -,-\rangle$
on $V$. Let us call a nondegenerate symmetric bilinear form a metric for short.
The induced  $G$-invariant metric $\langle-,-\rangle^*$ on $V^*$
 is canonically extended to a metric on $TV^*$ by the canonical identification
$V^*\cong TV^*$.
For the $1$-forms $dx^{\alpha},dx^{\beta}$ ($1\leq \alpha,\beta\leq n$), we put
\begin{equation}\nonumber
\tilde{\omega}^{\alpha\beta}=\langle dx^{\alpha},dx^{\beta}\rangle^*.
\end{equation}
This is a $G$-invariant homogeneous polynomial of degree $d_{\alpha}+d_{\beta}-2<2d_1$.
So  we can write it as follows.
\begin{equation}\nonumber
\tilde{\omega}^{\alpha\beta}=\tilde{\eta}^{\alpha\beta}x^1+\upsilon^{\alpha\beta},
\quad
\tilde{\eta}^{\alpha\beta},\upsilon^{\alpha\beta}\in \mathbb{C}[x']~.
\end{equation}
The matrix $\tilde{\eta}=(\tilde{\eta}^{\alpha\beta})$ is nondegenerate \cite[Corollary to (1.11)]{SaitoSekiguchiYano}.  
Let us  set
$$
\eta=\tilde{\eta}^{-1},\quad 
\eta=(\eta_{\alpha\beta})~.
$$
In \cite[Theorem 5.5]{KM2020},  it was proved that this  metric $\eta$ is compatible with the natural Saito structure mentioned in the previous paragraph.
This means that  $C_{\alpha\beta}^{\gamma}$'s  defined by \eqref{def-C} has the following property :
\begin{equation}\label{C-symmetry}
\sum_{\lambda=1}^n C_{\alpha\beta}^{\lambda}\eta_{\lambda \gamma}
=\sum_{\lambda=1}^n C_{\alpha\gamma}^{\lambda}\eta_{\lambda\beta}
\quad (1\leq \alpha,\beta,\gamma\leq n)~.
\end{equation}
This symmetry implies the existence of the potential function \cite{Dubrovin1993-2}.
We obtain an expression of the potential function (Theorem \ref{thm:potential}).
In \cite[(6.36)]{Satake2020}, Satake obtained an expression of $C_{\alpha\beta}^{\gamma}$ for irreducible  finite Coxeter groups.
Using \eqref{C-symmetry}, we  show that  
our expression in Theorem \ref{main-theorem2} agrees with 
his result (Theorem \ref{C-Satake}).

%%%%%%%%%%%%%%%%
\section{Preliminary II}
\label{preliminary-vanishing}
%%%%%%%%%%%%%%%%
In this section, we  do not  need the setting of finite  reflection  groups. So we just assume that
$d_1>d_2\geq \ldots \geq d_n\geq 2$ are  integers, 
$q\in \mathbb{C}^n$,
and that $z^1,\ldots,z^n$ are linear coordinates  of $\mathbb{C}^n$.
$M(n,\mathbb{C}(z))_q $ denotes the set of $n\times n$ matrices whose  entries  
are rational  functions in $z$ with  denominators which do not vanish at $q$.

\begin{definition}\label{def-condition-PQ}
Let $i=0,1$ and let $h\in \mathbb{Z}_{\geq 1}$.
\begin{itemize}
\item We say that $X=(X^{\gamma}_{\beta})\in M(n,\mathbb{C}(z))_q$ satisfies the condition $P(i)$
 if  $X(q)$ is diagonal and  if the following holds
for all $1\leq \beta,\gamma\leq n$ and all $a\in \mathbb{Z}_{\geq 0}^n\setminus\{0\}$:
$$
\text{if}~~\frac{\partial^a X^{\gamma}_{\beta}}{\partial z^a}(q)\neq 0
~~\text{then}~~
a\cdot d+d_{\beta}=d_{\gamma}+kd_1 ~~(k\in \mathbb{Z}_{\geq i})~.
$$
\item We say that  $X=(X^{\gamma}_{\beta})\in M(n,\mathbb{C}(z))_q$ satisfies the condition $Q(i,h)$
if
  the following holds
for all $1\leq \beta,\gamma\leq n$ and all $a\in \mathbb{Z}_{\geq 0}^n$:
$$
\text{if}~~\frac{\partial^a X^{\gamma}_{\beta}}{\partial z^a}(q)\neq 0
~~\text{then}~~
a\cdot d+h+d_{\beta}=d_{\gamma}+kd_1 ~~(k\in \mathbb{Z}_{\geq i})~.
$$
\end{itemize}
\end{definition}

\begin{lemma}\label{product-vanishing}
\begin{enumerate}
\item \label{product-vanishing-1}
Let $i=0,1$.
If $X,Y\in M(n,\mathbb{C}(z))_q$ satisfy $P(i)$, then the product $XY$ satisfies $P(i)$.
\item \label{product-vanishing-2}
Let $i,j=0,1$ and $h\in \mathbb{Z}_{\geq 1}$.
If $X\in M(n,\mathbb{C}(z))_q$ satisfies $P(i)$
and $Z\in M(n,\mathbb{C}(z))_q$ satisfies $Q(j,h)$,
the product $XZ$ satisfies $Q(j,h)$.
\item 
Let $h\in \mathbb{Z}_{\geq 1}$ and $1\leq \beta,\gamma\leq n$.
If $X\in M(n,\mathbb{C}(z))_q$ satisfies $P(1)$,
$Z\in M(n,\mathbb{C}(z))_q$ satisfies $Q(1,h)$,
and $a\in \mathbb{Z}_{\geq 0}^n$ satisfies
$a\cdot d+h+d_{\beta}=d_1+d_{\gamma}$,
then it holds that
$$
\frac{\partial^a }{\partial z^a}(XZ)(q)
=X(q)\frac{\partial^a Z}{\partial  z^a}(q)~.
$$
\item 
Let $i,j=0,1$ and $h\in \mathbb{Z}_{\geq 1}$.
If $X,Y\in M(n,\mathbb{C}(z))_q$ satisfy $P(i)$
and $Z\in M(n,\mathbb{C}(z))_q$ satisfies $Q(j,h)$,
the product $XZY$ satisfies $Q(j,h)$.
\item
Let $h\in \mathbb{Z}_{\geq 1}$  and $1\leq \beta,\gamma\leq n$.
If $X,Y\in M(n,\mathbb{C}(z))_q$ satisfy $P(1)$,
$Z\in M(n,\mathbb{C}(z))_q$ satisfies $Q(1,h)$,
and $a\in \mathbb{Z}_{\geq 0}^n$ satisfies
$a\cdot d+h+d_{\beta}=d_1+d_{\gamma}$,
then it holds that
$$
\frac{\partial^a }{\partial z^a}(XZY)(q)
=X(q)\frac{\partial^a Z}{\partial  z^a}(q)
Y(q)~.
$$
\end{enumerate}
\end{lemma}
\begin{proof}
We show (2)(3).  The proofs of (1)(4)(5) are similar and omitted.
\\
(2) For $a\in \mathbb{Z}_{\geq 0}^n$, we have
\begin{equation}\nonumber
\frac{\partial^a}{\partial z^a}(XZ)^{\gamma}_{\beta}(q)
=\sum_{\rho=1}^n 
\sum_{\begin{subarray}{c}a=b+c\\b\neq 0\end{subarray}}
\frac{a!}{b!c!}
\underbrace{\frac{\partial^b X^{\gamma}_{\rho}}{\partial z^b}(q)\cdot
\frac{\partial^c Z^{\rho}_{\beta}}{\partial z^c}(q)}_{(\star)}
+X^{\gamma}_{\gamma}(q) \underbrace{\frac{\partial^a Z^{\gamma}_{\beta}}{\partial z^a}(q)}_{(\star\star)}~.
\end{equation}
If $(\star)\neq 0$, 
$$
b\cdot d+d_{\rho}=d_{\gamma}+k_1 d_1~~(k_1\in \mathbb{Z}_{\geq i})
,\quad 
c\cdot d+h+d_{\beta}=d_{\rho}+k_2 d_1~~(k_2\in \mathbb{Z}_{\geq j})
$$
since $X$ and $Z$ satisfy $P(i)$ and $Q(j,h)$ respectively.
Taking the sum of these two equations, we obtain
\begin{equation}\label{vanishing-star1}
a\cdot d+h+d_{\beta}=d_{\gamma}+kd_1~~(k\in \mathbb{Z}_{\geq i+j})~.
\end{equation}
If $(\star\star)\neq 0$, we have
$$
a\cdot d+h+d_{\beta}=d_{\gamma}+kd_1~~(k\in \mathbb{Z}_{\geq j})~.
$$
If $\frac{\partial^a}{\partial z^a}(XZ)^{\gamma}_{\beta}(q)\neq 0$,
then at least one of $(\star)$'s is not zero, or $(\star\star)\neq 0$.
This implies that 
$$
a\cdot d+h+d_{\beta}=d_{\gamma}+kd_1~~(k\in \mathbb{Z}_{\geq j})
$$
since $\mathbb{Z}_{\geq j}\supseteq \mathbb{Z}_{\geq i+j}$.
\\
(3) When $i=j=1$,  \eqref{vanishing-star1} means that 
$(\star)=0$ if $a\cdot d+h+d_{\beta}=d_{\gamma}+d_1$.
\end{proof}

\begin{lemma}\label{inverse-vanishing}
Assume that $X\in M(n,\mathbb{C}(z))_q$ is invertible in a neighborhood of $q\in V$.
We set $\tilde{X}=X^{-1}$.
\begin{enumerate}
\item Let $i=0,1$.
If $X$ satisfies $P(i)$,  So is $\tilde{X}$.
\item
If $X$ satisfies $P(1)$  and 
$a\in  \mathbb{Z}_{\geq 0}^n\setminus\{0\}$ satisfies $a\cdot d+d_{\beta}=d_1+d_{\gamma}$, it holds that
\begin{equation} \nonumber 
\frac{\partial \tilde{X}^{\gamma}_{\beta}}{\partial z^{a}}(q)
=-\frac{1}{X^{\gamma}_{\gamma}(q)X^{\beta}_{\beta}(q)}
\frac{\partial X^{\gamma}_{\beta}}{\partial z^{a}}(q)
\quad (1\leq \beta,\gamma\leq n)~.
\end{equation}
\end{enumerate}
\end{lemma}

\begin{proof}
(1)
Since $X(q)$ is diagonal,  it is clear that the inverse matrix $\tilde{X}(q)$
is also diagonal.
So we show the following statement by  induction on $|a|$.
$$
\text{If}~~\frac{\partial^a \tilde{X}^{\gamma}_{\beta}}{\partial z^a}
(q)\neq 0
~~\text{then}~~
a\cdot d+d_{\beta}=d_{\gamma}+kd_1\quad (k\in \mathbb{Z}_{\geq i}),
$$
where $a\in \mathbb{Z}_{\geq 0}^n\setminus\{0\}$.
 
First we show the statement for $|a|=1$, i.e. $a=\bm{e}_{\alpha}$ for some $1\leq \alpha\leq n$.
Differentiating by $z^{\alpha}$ the equation
\begin{equation}\label{inverse-X}
\sum_{\rho=1}^n \tilde{X}^{\gamma}_{\rho} X^{\rho}_{\beta}=\delta^{\gamma}_{\beta}
\end{equation}
and evaluating at $q$,
we obtain 
$$
\frac{\partial \tilde{X}^{\gamma}_{\beta}}{\partial z^{\alpha}}(q)X^{\beta}_{\beta}(q)
+\tilde{X}^{\gamma}_{\gamma}(q)\frac{\partial X^{\gamma}_{\beta}}{\partial z^{\alpha}}(q)
=0~.$$
Therefore 
$\frac{\partial \tilde{X}^{\gamma}_{\beta}}{\partial z^{\alpha}}(q)\neq 0$
implies $\frac{\partial X^{\gamma}_{\beta}}{\partial z^{\alpha}}(q)\neq 0$.
Thus the statement holds when $|a|=1$. 
 
Next assume that the statement holds for 
$1\leq |a|<m$ and take $a\in \mathbb{Z}_{\geq 0}^n\setminus\{0\}$
with $|a|=m$.
Differentiating \eqref{inverse-X} by $z^a$ and evaluating at $q$, we obtain
$$
\frac{\partial^a \tilde{X}^{\gamma}_{\beta}}{\partial z^a}(q) X^{\beta}_{\beta}(q)=-
\sum_{\rho=1}^n 
\sum_{\begin{subarray}{c}a=b+c\\b\neq 0\\c\neq 0\end{subarray}}
\frac{a!}{b!c!}
\underbrace{\frac{\partial^b \tilde{X}^{\gamma}_{\rho}}{\partial z^b}(q)
\frac{\partial^c {X}^{\rho}_{\beta}}{\partial z^c}(q)}_{(\star)}
-\tilde{X}^{\gamma}_{\gamma}(q)\underbrace{\frac{\partial^a X^{\gamma}_{\beta}}{\partial z^a}(q)}_{(\star\star)}
~.
$$
If $(\star)\neq 0$, we have
\begin{equation}\nonumber
b\cdot d+d_{\rho}=d_{\gamma}+k_1d_1~~(k_1\in \mathbb{Z}_{\geq i}),
\quad
c\cdot d+d_{\beta}=d_{\rho}+k_2d_1~~(k_2\in \mathbb{Z}_{\geq i}),
\end{equation}
by the assumption of the induction and the assumption on $X$.
Summing up these two equations, we have
\begin{equation}\label{vanishing-star2}
a\cdot d+d_{\beta}=d_{\gamma}+kd_1~~(k\in \mathbb{Z}_{\geq 2i})~.
\end{equation}
If $(\star\star)\neq 0$, we have
$$
a\cdot d+d_{\beta}=d_{\gamma}+kd_1~~(k\in \mathbb{Z}_{\geq i})~,
$$
by the assumption on $X$.
$\frac{\partial^a \tilde{X}^{\gamma}_{\beta}}{\partial z^a}(q)\neq 0$
implies that 
at least one of $(\star)$'s is not zero or
$(\star\star)\neq 0$.
Given that $\mathbb{Z}_{\geq 2i}\subseteq \mathbb{Z}_{\geq i}$,
we obtain the statement for $|a|=m$.
\\
(2) Eq.\eqref{vanishing-star2} means that 
$(\star)=0$ if $i=1$ and if  $a\cdot d+d_{\beta}=d_{\gamma}+d_1$.
\end{proof}

\begin{remark}\label{mathcalZ}
The statements of Lemma \ref{product-vanishing} and Lemma \ref{inverse-vanishing} hold if we replace
``$a\in \mathbb{Z}_{\geq 0}^n$'' and ``$a\in \mathbb{Z}_{\geq 0}^n\setminus\{0\}$''  
by ``$a\in\mathcal{Z}$'' and $a\in \mathcal{Z}\setminus\{0\}$
in Definition \ref{def-condition-PQ}, Lemma \ref{product-vanishing} and Lemma \ref{inverse-vanishing}.
\end{remark}

%%%%%%%%%%%%%%%%%%%%%%%%%%%%%%%%%%
\section{Good basic invariants and natural Saito structures}
\label{sec:duality-groups}
%%%%%%%%%%%%%%%%%%%%%%%%%%%%%%%%%%
In this section,  $G$  is a  duality  group,
$x$ is a  set  of good basic invariants with  respect to an admissible triplet
$(g,\zeta,q)$ and $z$  is a $(g,\zeta)$-graded coordinate  system of  $V$.
It is also assumed that $x$ is compatible with $z$ at $q$.

%%%%%%%%%%%%%%%%%%%%%%%%%%%%%%%%%%%%%%%
\subsection{Vanishing Properties of  $J,\tilde{J}$, $\Omega_{\alpha}$} 
\label{vanishings}
%%%%%%%%%%%%%%%%%%%%%%%%%%%%%%%%%%%%%%%
First
we apply 
Lemma \ref{inverse-vanishing} to the Jacobian matrix $J=(\frac{\partial x^{\gamma}}{\partial z^{\beta}})$ 
and study the vanishing at $q$  of  derivatives of the inverse  matrix $\tilde{J}$.
\begin{lemma}\label{lemma2}
\begin{enumerate}
\item For $1\leq \beta,\gamma\leq n$,
$$
J^{\gamma}_{\beta}(q)=\tilde{J}^{\gamma}_{\beta}(q)=\delta^{\gamma}_{\beta}~.$$
\item
 For $1\leq \beta,\gamma\leq n$ and 
for $a\in \mathbb{Z}_{\geq 0}^n\setminus\{0\}$, 
$$
\text{if }~~\frac{\partial^aJ^{\gamma}_{\beta}}{\partial z^a} (q)
\neq 0 \quad \text{then}\quad
a\cdot d+d_{\beta}=d_{\gamma}+kd_1 ~~(k\in \mathbb{Z}_{\geq 1}),
$$
\item \label{lemma2-3}
 For  $1\leq \beta,\gamma\leq n$ and 
 for $a\in \mathbb{Z}_{\geq 0}^n\setminus\{0\}$,
$$
\text{if }~~\frac{\partial^a\tilde{J}^{\gamma}_{\beta}}{\partial z^a} (q)
\neq 0 \quad \text{then}\quad
a\cdot d+d_{\beta}=d_{\gamma}+kd_1 ~~(k\in \mathbb{Z}_{\geq 1}).
$$
\item
 For $1\leq \beta,\gamma\leq n$ and for
$a\in \mathbb{Z}_{\geq 0}^n\setminus\{0\}$ satisfying
$d_{\beta}+a\cdot d=d_{\gamma}+d_1$, it holds that
$$
\frac{\partial^a\tilde{J}^{\gamma}_{\beta}}{\partial z^a} (q)=
-\frac{\partial^aJ^{\gamma}_{\beta}}{\partial z^a} (q)~.
$$
\end{enumerate}\end{lemma}
\begin{proof}
 (1)  is  a consequence of 
the compatibility condition \eqref{compatible-x0}.
The statement (2) for $J^{\gamma}_{\beta}$
is nothing but Lemma \ref{x-derivatives}.
The results of (1) and (2) imply that
the matrix $J=(J^{\gamma}_{\beta})$ satisfies the condition $P(1)$.
Then (3) and (4) follow from  Lemma \ref{inverse-vanishing}.
 \end{proof}

Next we study the vanishing at $q$  of derivatives of $\Omega_{\alpha}$.
Let us set
\begin{equation}\label{defL}
L_{\alpha\beta}^{\gamma}=\frac{\partial \tilde{J}^{\gamma}_{\alpha}}{\partial z^{\beta}}
\quad (1\leq \alpha,\beta,\gamma\leq n)~.
\end{equation}
The $n\times n$ matrix whose $(\gamma,\beta)$-entries are 
$L_{\alpha\beta}^{\gamma}$ is denoted $L_{\alpha}$.
Notice that $\Omega_{\alpha}$ and $L_{\alpha}$ are related as follows.
\begin{equation}\nonumber 
\begin{split}
\Omega_{\alpha\beta}^{\gamma}&=
-\sum_{\nu=1}^n\left(\frac{\partial J}{\partial z^{\nu}}\tilde{J}\right)^{\gamma}_{\alpha}~
\tilde{J}^{\nu}_{\beta}
=\sum_{\nu=1}^ n\left(J\frac{\partial \tilde{J}}{\partial z^{\nu}}\right)^{\gamma}_{\alpha}
~
\tilde{J}^{\nu}_{\beta}
=\sum_{\lambda,\nu=1}^n 
J^{\gamma}_{\lambda}\frac{\partial \tilde{J}^{\lambda}_{\alpha}}{\partial z^{\nu}}\tilde{J}^{\nu}_{\beta}
\\&=(JL_{\alpha}\tilde{J})^{\gamma}_{\beta}~.
\end{split}
\end{equation}

\begin{lemma}\label{lemma3}
\begin{enumerate}
\item For $1\leq \alpha,\beta,\gamma\leq n$ and
$a\in \mathbb{Z}^n_{\geq 0}$,
$$
\text{if}\quad \frac{\partial^aL_{\alpha\beta}^{\gamma}}{\partial  z^a}(q)\neq 0\quad \text{then}\quad 
a\cdot d+d_{\alpha}+d_{\beta}=d_{\gamma}+kd_1 ~~(k\in \mathbb{Z}_{\geq 1})~.
$$
\item For $1\leq \alpha,\beta,\gamma\leq n$ and
$a\in \mathbb{Z}^n_{\geq 0}$,
$$
\text{if}\quad \frac{\partial^a\Omega_{\alpha\beta}^{\gamma}}{\partial  z^a}(q)\neq 0\quad \text{then}\quad 
a\cdot d+d_{\alpha}+d_{\beta}=d_{\gamma}+kd_1 ~~(k\in\mathbb{Z}_{\geq 1})~.
$$
\item For $1\leq \beta,\gamma\leq n$.
$$
\Omega_{1\beta}^{\gamma}(q)
=\frac{1-d_{\gamma}}{z^1(q)}\delta^{\gamma}_{\beta}
$$
\end{enumerate}
\end{lemma}

\begin{proof}
(1) The first statement is an immediate consequence of 
Lemma \ref{lemma2} \eqref{lemma2-3}.
\\
(2) Lemma \ref{lemma2} (1)(2)  imply that
$J,\tilde{J}$ satisfy the condition $P(1)$.
The result (1) implies that $L_{\alpha}$ satisfies the condition $Q(1,d_{\alpha})$.
Therefore $\Omega_{\alpha}=JL_{\alpha}\tilde{J}$
satisfies the condition $Q(1,d_{\alpha})$ by Lemma \ref{product-vanishing} (4).
\\
(3) 
Evaluating \eqref{Omega} at $q$, we have
$$
d_1 x^1(q)\Omega_{1\beta}^{\gamma}(q)=(1-d_{\gamma})\delta^{\gamma}_{\beta}~.
$$
Since $x^1(q)=\frac{z^1(q)}{d_1}$  (see Lemma \ref{compatible-x-form}),  
the statement follows.
\end{proof}

%%%%%%%%%%%%%%%%%%%%%%%%%%%%%%%%%%%%
\subsection{Vanishing of  $\Gamma_{\alpha\beta}^{\gamma}$}
%%%%%%%%%%%%%%%%%%%%%%%%%%%%%%%%%%%%
Recall the  definition  of $\Gamma_{\alpha\beta}^{\gamma}\in \mathbb{C}[x']
$ ($1\leq  \alpha,\beta,\gamma\leq n$):
it is the coefficient of $(x^1)^n$ in $\Delta\Omega_{\alpha\beta}^{\gamma}$.
\begin{theorem}\label{good-implies-flat}
$\Gamma_{\alpha\beta}^{\gamma}=0$ 
for $1\leq \alpha,\beta,\gamma\leq n$.
\end{theorem}

\begin{proof}
Fix $\alpha,\beta,\gamma$.
We take $a\in \mathcal{Z}$ satisfying $a\cdot d=d_{\gamma}-d_{\alpha}-d_{\beta}$
and 
differentiate $\Delta \Omega_{\alpha\beta}^{\gamma}$ by $z^a$
and evaluate at $q$.  We have
$$
\frac{\partial^a}{\partial z^a}
(\Delta \Omega_{\alpha\beta}^{\gamma})(q)
=\sum_{a'+a''=a}
\frac{a!}{a'!a''!}
\frac{\partial^{a'} \Delta}{\partial z^{a'}}(q)\cdot 
\underbrace{\frac{\partial^{a''} \Omega_{\alpha\beta}^{\gamma}}{\partial z^{a''}}(q)}_{(\star)}.
$$
Applying Lemma \ref{lemma3} (2), we see that 
$(\star)=0$.  Therefore
$$
\frac{\partial^a}{\partial z^a}
(\Delta \Omega_{\alpha\beta}^{\gamma})(q)=0~.
$$
On the other hand,
given that the degree of
$\Delta\Omega_{\alpha\beta}^{\gamma} \in \mathbb{C}[x]$
is $nd_1+d_{\gamma}-d_{\alpha}-d_{\beta}$,
it is written as follows:
\begin{equation}\label{DeltaOmega}
\Delta \Omega_{\alpha\beta}^{\gamma}
=\sum_{k=0}^n\left( \sum_{\begin{subarray}{c}
b\in \mathcal{Z}\\
b\cdot d=d_{\gamma}-d_{\alpha}-d_{\beta}+kd_1
\end{subarray}}
A_b^{(k)}x^b \right)\cdot (x^1)^{n-k}~.
\end{equation}
We again take $a\in \mathcal{Z}$ satisfying $a\cdot d=d_{\gamma}-d_{\alpha}-d_{\beta}$.
Differentiating  \eqref{DeltaOmega} by $z^a$
and evaluating at $q$, we obtain
$$
0=\sum_{k=0}^n
\sum_{\begin{subarray}{c}
b\in \mathcal{Z}\\
b\cdot d=d_{\gamma}-d_{\alpha}-d_{\beta}+kd_1
\end{subarray}}
A_b^{(k)}
\sum_{a=a'+a''}
\frac{a!}{a'!a''!}
\underbrace{\frac{\partial^{a'}x^b}{\partial z^{a'}}(q)}_{(\star\star)}\cdot
\frac{\partial^{a''}(x^1)^{n-k}}{\partial z^{a''}}(q)~.
$$
By Lemma \ref{a-b-improved}, 
$(\star\star)\neq 0$ holds only if 
\begin{enumerate}
\item[(i)]
$|a'|>|b|$ and $a'\cdot d=b\cdot d+ld_1$ ($l\geq 1$), or
\item[(ii)] $a'=b$
\end{enumerate}
Since $a'\cdot d\leq a\cdot d=d_{\gamma}-d_{\alpha}-d_{\beta}\leq b\cdot d,$
Case (i) cannot happen.
Case (ii) occur only when $a'=a=b$ and $b\cdot d=d_{\gamma}-d_{\alpha}-d_{\beta}$.
Therefore 
\begin{equation}\nonumber
\begin{split}
0&=
A_a^{(0)}\frac{\partial^{a} x^a}{\partial z^a} (q)\cdot
x^1(q)^n=
a! A_a^{(0)}x^1(q)^n~.
\end{split}
\end{equation}
Here we used  Lemma \ref{a-b} (2).
Since $x^1(q)\neq 0$, we have $A_a^{(0)}=0$.

Notice that $\Gamma_{\alpha\beta}^{\gamma}$ is the $k=0$ part of \eqref{DeltaOmega}, that is,
$$\Gamma_{\alpha\beta}^{\gamma}
= \sum_{\begin{subarray}{c}
b\in \mathcal{Z}\\
b\cdot d=d_{\gamma}-d_{\alpha}-d_{\beta}
\end{subarray}}
A_b^{(0)}x^b=0~.
$$
\end{proof}

%%%%%%%%%%%%%%%%%%%%%%%%%%%%%%%%%%
\subsection{The inverse $\tilde{\Omega}$ of $\Omega_1$}
%%%%%%%%%%%%%%%%%%%%%%%%%%%%%%%%%%
\begin{lemma}\label{lemma33}
\begin{enumerate}
\item 
For $1\leq \beta,\gamma\leq n$ and for
$a\in \mathcal{Z}\setminus\{0\}$,
$$
\text{if}\quad \frac{\partial^a\tilde{\Omega}_{\beta}^{\gamma}}{\partial  z^a}(q)\neq 0\quad \text{then}\quad 
a\cdot d+d_{\beta}=d_{\gamma}+kd_1 ~~(k\in \mathbb{Z}_{\geq 1})~.
$$

\item For $1\leq \alpha,\beta,\gamma\leq n$ and
for $a\in  \mathcal{Z}\setminus\{0\}$ satisfying
$a\cdot d=d_1+d_{\gamma}-d_{\alpha}-d_{\beta}$, 
$$
 \frac{\partial^a}{\partial  z^a}
(\tilde{\Omega}\Omega_{\alpha})^{\gamma}_{\beta}(q)
=\frac{z^1(q)}{d_{\gamma}-1}
\frac{\partial^a}{\partial z^a}\left(\frac{\partial^2 x^{\gamma}}{\partial z^{\alpha}\partial z^{\beta}}\right)(q)~.
$$
\end{enumerate}
\end{lemma}
\begin{proof}
(1) 
Given \eqref{D1-Gamma}, \eqref{invOmega} and $\Gamma_{\alpha\beta}^{\gamma}=0$, $\tilde{\Omega}^{\gamma}_{\beta}$
is written as follows:
$$
\tilde{\Omega}_{\beta}^{\gamma}
=\frac{d_1\delta^{\gamma}_{\beta}}{1-d_{\gamma}}
x^1
+\sum_{\begin{subarray}{c}
b\in \mathcal{Z}\\b\cdot d=d_1+d_{\gamma}-d_{\beta}
\end{subarray}}
B_b x^b \quad (B_b\in \mathbb{C})~.
$$
Take $a\in \mathcal{Z}\setminus\{0\}$ and  consider the derivative by $z^a$:
$$
 \frac{\partial^a\tilde{\Omega}_{\beta}^{\gamma}}{\partial  z^a}(q)
 =\underbrace{\frac{d_1\delta^{\gamma}_{\beta}}{1-d_{\gamma}}
 \frac{\partial^a x^1}{\partial z^a}(q)}_{(\star)}
+\sum_{\begin{subarray}{c}
b\in \mathcal{Z}\\b\cdot d=d_1+d_{\gamma}-d_{\beta}
\end{subarray}}
B_b \underbrace{\frac{\partial^a x^b}{\partial z^a}(q)}_{(\star\star)}~.
$$
Assume that $\frac{\partial^a\tilde{\Omega}_{\beta}^{\gamma}}{\partial  z^a}(q)\neq 0$.
Then $(\star)\neq 0$ or at least one of $(\star\star)$'s is not zero.
If $(\star)\neq 0$,
$a\cdot d=kd_1$ ($k\in \mathbb{Z}_{\geq 2}$)
by Lemma \ref{x-derivatives}.
The factor $\delta^{\gamma}_{\beta}$ means $d_{\beta}=d_{\gamma}$,
so  $a\cdot d+d_{\beta}=d_{\gamma}+kd_1$ ($k\in \mathbb{Z}_{\geq 2}$) holds.
If $(\star\star)\neq 0$,  we have
$a\cdot d=b\cdot d+kd_1=d_{\gamma}-d_{\beta}+(k+1)d_1$ ($k\in \mathbb{Z}_{\geq 1}$)
or $a=b$ by Lemma \ref{a-b-improved}.
In the both cases, $a\cdot d=d_{\gamma}-d_{\beta}+k'd_1$ ($k'\in \mathbb{Z}_{\geq 1}$) holds.
\\
(2) The result of (1)  implies that $\tilde{\Omega}$ satisfies the condition
$P(1)$ with $\mathbb{Z}_{\geq 0}^n$ replaced by $\mathcal{Z}$.
Lemma \ref{lemma3} (2) implies that 
$\Omega_{\alpha}$ satisfies the condition $Q(1,d_{\alpha})$.
Therefore
by Lemma \ref{product-vanishing} (3) and Remark \ref{mathcalZ},
we have
$$
 \frac{\partial^a}{\partial  z^a}
(\tilde{\Omega}\Omega_{\alpha})(q)
=\tilde{\Omega}(q)
 \frac{\partial^a \Omega_{\alpha}}{\partial  z^a}(q)=
 \frac{z^1(q)}{1-d_{\gamma}} \frac{\partial^a \Omega_{\alpha}}{\partial  z^a}(q)
$$
 for $a\in \mathcal{Z}$ satisfying 
$a\cdot d=d_1+d_{\gamma}-d_{\alpha}-d_{\beta}$.
Further, we apply Lemma \ref{product-vanishing} (5) to
$$
 \frac{\partial^a \Omega_{\alpha}}{\partial  z^a}(q) 
 =\frac{\partial^a}{\partial  z^a}(JL_{\alpha}\tilde{J})(q)~.
$$
Lemma \ref{lemma3} (1) implies that $L_{\alpha}$ satisfies the condition $Q(1,d_{\alpha})$. 
Lemma \ref{lemma2} (1)(2) imply that
$J,\tilde{J}$ satisfy the condition $P(1)$.
Therefore we have
$$
\frac{\partial^a}{\partial  z^a}(JL_{\alpha}\tilde{J})(q)
=J(q) \frac{\partial^a L_{\alpha}}{\partial  z^a}(q) \tilde{J}(q)
=\frac{\partial^a L_{\alpha}}{\partial  z^a}(q)~.
$$
Recalling the definition \eqref{defL} of $L_{\alpha}$, we have
\begin{equation}\nonumber 
\frac{\partial^a L_{\alpha\beta}^{\gamma}}{\partial z^a}(q)=
\frac{\partial^a}{\partial z^a}\left(
\frac{\partial \tilde{J}^{\gamma}_{\alpha}}{\partial z^{\beta}}
\right)(q)~.
\end{equation}
Then applying  Lemma \ref{lemma2} (4),  we obtain
\begin{equation}\label{L-q-2}
\frac{\partial^a L_{\alpha\beta}^{\gamma}}{\partial z^a}(q)
=-\frac{\partial^a}{\partial z^a}\left(
\frac{\partial {J}^{\gamma}_{\alpha}}{\partial z^{\beta}}
\right)(q)
=-\frac{\partial^a}{\partial z^a}\left(\frac{\partial^2 x^{\gamma}}{\partial z^{\alpha}\partial z^{\beta}}\right)(q)~.
\end{equation}
This completes the proof.
\end{proof}

%%%%%%%%%%%%%%%%%%%%%%%%%%%%%%%%
\subsection{Expression of $C_{\alpha\beta}^{\gamma}$}
%%%%%%%%%%%%%%%%%%%%%%%%%%%%%%%%
\begin{theorem}\label{main-theorem2}
It holds that
$$
C_{\alpha\beta}^{\gamma}=\frac{z^1(q)}{d_{\gamma}-1}
\sum_{\begin{subarray}{c}
a\in \mathcal{Z}\\
a\cdot d=d_1+d_{\gamma}-d_{\alpha}-d_{\beta}
\end{subarray}}
\frac{\partial^a}{\partial z^a}
\left(\frac{\partial^2 x^{\gamma}}{\partial z^{\alpha}\partial z^{\beta}}\right)
(q)\cdot \frac{x^a}{a!}~.
$$
\end{theorem}
\begin{proof}
Fix $1\leq \alpha,\beta,\gamma\leq n$.
As  we  proved in Theorem \ref{good-implies-flat}, $\Gamma_{\alpha}=O$ holds.
Substituting this into  \eqref{def-C},  we have
$$
C_{\alpha\beta}^{\gamma}=(\tilde{\Omega}\Omega_{\alpha})^{\gamma}_{\beta}~.
$$
We take $a\in \mathcal{Z}$ satisfying 
$a\cdot d=d_1+d_{\gamma}-d_{\alpha}-d_{\beta}$
and differentiate $C_{\alpha\beta}^{\gamma}$ by $z^a$.
Then by Lemma \ref{lemma33} (2), 
\begin{equation}\nonumber
\frac{\partial^a C_{\alpha\beta}^{\gamma}}{\partial z^a}(q)=
 \frac{\partial^a}{\partial  z^a}
(\tilde{\Omega}\Omega_{\alpha})^{\gamma}_{\beta}(q)
=\frac{z^1(q)}{d_{\gamma}-1}
\frac{\partial^a}{\partial z^a}
\left(\frac{\partial^2 x^{\gamma}}{\partial z^{\alpha}\partial z^{\beta}}\right)
(q)~.
\end{equation}
On  the  other hand,
given that $\mathrm{deg}\, C_{\alpha\beta}^{\gamma}=d_1+d_{\gamma}-d_{\alpha}-d_{\beta}$, $C_{\alpha\beta}^{\gamma}$ is written as follows.
$$
C_{\alpha\beta}^{\gamma}
=\sum_{\begin{subarray}{c}
b\in \mathcal{Z}\\
b\cdot d=d_1+d_{\gamma}-d_{\alpha}-d_{\beta}
\end{subarray}}
A_b x^b  \quad  (A_b\in \mathbb{C})~.
$$
By Lemma \ref{a-b-improved}, we have
\begin{equation}\nonumber
\frac{\partial^a C_{\alpha\beta}^{\gamma}}{\partial z^a}(q)=
\sum_{\begin{subarray}{c}
b\in \mathcal{Z}\\
b\cdot d=d_1+d_{\gamma}-d_{\alpha}-d_{\beta}
\end{subarray}}
A_b \frac{\partial^a x^b}{\partial z^a}(q) =a!A_a~.
\end{equation}
Comparing these two equations, we obtain $A_a$.
\end{proof}
 
By definition, we have $C_1=(C_{1\beta}^{\gamma})=\tilde{\Omega}\Omega_1=I$.
Combining with Theorem \ref{main-theorem2},
we have the following
\begin{corollary}\label{C1-vanishing}
For  $a\in \mathcal{Z}$ satisfying
$a\cdot d=d_{\gamma}-d_{\beta}$, it holds that
\begin{equation}\nonumber
\frac{\partial^a}{\partial z^a}
\frac{\partial^2 x^{\gamma}}{\partial z^1 \partial z^{\beta}}(q)
=
\begin{cases} 0&(\gamma\neq \beta)\\
\frac{d_{\gamma}-1}{z^1(q)}&(\gamma=\beta)
\end{cases}~.
\end{equation}
\end{corollary}
The proof is immediate and omitted.   The result for the case $\gamma=\beta$ in Corollary \ref{C1-vanishing} also follows from Lemma \ref{compatible-x-form}.
We will  use this result in the proof of Corollary \ref{vector-potential}.

%%%%%%%%%%%%%%%%%%%%
\subsection{Potential vector  field}
%%%%%%%%%%%%%%%%%%%%
We put
\begin{equation}\label{eq:pvf}
\mathcal{G}^{\gamma}=
\frac{z^1(q)}{d_{\gamma}-1}
\sum_{a\in I_{\gamma}^{(1)}} 
\frac{\partial^a x^{\gamma}}{\partial z^a} (q) \cdot \frac{x^a}{a!}
\quad (1\leq \gamma\leq n)~.
\end{equation}
Here $$
I^{(1)}_{\gamma}=\{a\in \mathbb{Z}^n_{\geq 0}\mid 
a\cdot d=d_{\gamma}+d_1, |a|\geq 2\}~,
$$
as defined  in  \eqref{defIalpha}.
\begin{corollary}\label{vector-potential}
It holds that
$$
C_{\alpha\beta}^{\gamma}=
\frac{\partial^2 \mathcal{G}^{\gamma}}{\partial x^{\alpha}\partial x^{\beta}}~.
$$
\end{corollary}

\begin{proof}
Differentiating $\mathcal{G}^{\gamma}$, we have
\begin{equation}\begin{split}\nonumber
\frac{\partial^2 \mathcal{G}^{\gamma}}{\partial x^{\alpha}\partial x^{\beta}}
&=
\frac{z^1(q)}{d_{\gamma}-1}
\sum_{\begin{subarray}{c}a\in I_{\gamma}^{(1)};\\ 
a_{\alpha},a_{\beta}\geq 1
\end{subarray} }
\frac{\partial^a x^{\gamma}}{\partial z^a} (q) 
\cdot
\frac{x^{a-\bm{e}_{\alpha}-\bm{e}_{\beta}}}{a!}\times 
\begin{cases}
a_{\alpha} a_{\beta}&  (\alpha\neq \beta)\\
a_{\alpha}(a_{\alpha}-1)&(\alpha=\beta)
\end{cases}
\end{split}
\end{equation}
If we write $b=a-\bm{e}_{\alpha}-\bm{e}_{\beta}$ in the RHS,
this becomes
$$
\frac{\partial^2 \mathcal{G}^{\gamma}}{\partial x^{\alpha}\partial x^{\beta}}=
\frac{z^1(q)}{d_{\gamma}-1}
\sum_{\begin{subarray}{c}
b\in \mathbb{Z}_{\geq 0}^n;\\
b\cdot d=d_{1}+d_{\gamma}-d_{\alpha}-d_{\beta}
\end{subarray}} 
\underbrace{\frac{\partial^b}{\partial z^b} 
\frac{\partial^2 x^{\gamma}}{\partial z^{\alpha}\partial z^{\beta}}(q) }_{(\star)}\cdot 
\frac{x^b}{b!}~.
$$
Comparing with the expression in Theorem \ref{main-theorem2}, 
we only have to show that terms  with $b\in \mathbb{Z}_{\geq 0}^ n\setminus \mathcal{Z}$ (i.e. $b$ such that $b_1\geq 1$) 
do not contribute to the sum  in the RHS.

When $\alpha=1$, $\beta=1$, $\alpha=\gamma$ or $\beta=\gamma$,
$b\in \mathbb{Z}_{\geq}^0$  satisfying $b\cdot d=d_1+d_{\gamma}-d_{\alpha}-d_{\beta}$ has  $b_1=0$. Indeed,
if $\alpha=1$, $b_1d_1\leq b\cdot d=d_{\gamma}-d_{\beta}<d_1$ implies
$b_1=0$.

When $\alpha,\beta\neq 1,\gamma$,
$b\in\mathbb{Z}^n_{\geq 0}$ satisfying $b_1\geq 1$ and 
$b\cdot d=d_1+d_{\gamma}-d_{\alpha}-d_{\beta}$  has $b_1=1$,
because
$b_1d_1\leq b\cdot d=d_1+d_{\gamma}-d_{\alpha}-d_{\beta}<2d_1$.
If  we put $a=b-\bm{e}_1+\bm{e}_{\alpha}$, then
$a\in  \mathcal{Z}$, $a\cdot d=d_{\gamma}-d_{\beta}$ and
\begin{equation}\label{exchange}
(\star)=\frac{\partial^a }{\partial z^a}
\frac{\partial^2 x^{\gamma}}{\partial z^{1}\partial z^{\beta}}(q)~.
\end{equation}
This vanishes  by Corollary \ref{C1-vanishing}.
\end{proof}

%%%%%%%%%%%%%%%%%%%%%%%%%%%%%%%
\section{The case of irreducible finite Coxeter groups}
\label{Satake-C}
%%%%%%%%%%%%%%%%%%%%%%%%%%%%%%%
In this section, $G$ is an irreducible finite Coxeter group acting on $V=\mathbb{C}^n$, and  $(g,\zeta,q)$ is an admissible triplet of $G$.
$z=(z^1,\ldots,z^n)$ is a $(g,\zeta)$-graded coordinate system  of $V$
and $x=(x^1,\ldots, x^n)$ is a good basic invariants with respect to
$(g,\zeta,q)$ which is compatible with $z$ at $q$.
As in  \S \ref{second-part},
$\langle-,-\rangle$ denotes the $G$-invariant metric  on $V$
and $\langle-,-\rangle^*$ denotes the dual metric  on $V^*$.

%%%%%%%%%%%
\subsection{Metric}
%%%%%%%%%%%
\begin{lemma}\label{metric1}
If $d_{\alpha}+d_{\beta}\neq d_1+2$,
$$
\langle z^{\alpha},z^{\beta}\rangle^* =0~.
$$
\end{lemma}
\begin{proof}
By the $G$-invariance of the metric, we have
$$ 
\langle z^{\alpha},z^{\beta}\rangle^*
=\langle g^*z^{\alpha},g^*z^{\beta}\rangle^*
=\zeta^{d_{\alpha}+d_{\beta}-2}  \langle z^{\alpha},z^{\beta}\rangle^*~.
$$
This implies that
$\langle z^{\alpha},z^{\beta}\rangle^* =0$ if 
$ d_{\alpha}+d_{\beta}\not \equiv 2\mod d_1$.
Since $2\leq d_{\alpha},d_{\beta}\leq d_1$, 
the condition $ d_{\alpha}+d_{\beta}\not \equiv 2\mod d_1$
is equivalent to $ d_{\alpha}+d_{\beta}\neq d_1+2$.
%which implies 
%$d_{\alpha}+d_{\beta}\neq d_1+2$.
\end{proof}
Given that $d_{\alpha}+d_{n+1-\alpha}=d_1+2$ holds,
$\langle z^{\alpha},z^{\beta}\rangle^*$ is block anti-diagonal.
If all $d_{\alpha}$'s are distinct, which is the case except for  $D_n$ with $n$ even, 
it is anti-diagonal. 

As we mentioned in \S \ref{second-part},  define $\tilde{\omega}^{\alpha\beta}$,
$\tilde{\eta}^{\alpha\beta}$ and $\eta_{\alpha\beta}$ by
$$
\tilde{\omega}^{\alpha\beta}=
\langle dx^{\alpha},dx^{\beta}\rangle^*,
\quad
\tilde{\omega}^{\alpha\beta}=\tilde{\eta}^{\alpha\beta}x^1+\upsilon^{\alpha\beta}\quad(
\tilde{\eta}^{\alpha\beta},\upsilon^{\alpha\beta}\in \mathbb{C}[x'])~,
\quad (\eta_{\alpha\beta})=(\tilde{\eta}^{-1})_{\alpha\beta}~.
$$
\begin{lemma}\label{tilde-omega-q}
\begin{enumerate}
\item $\tilde{\omega}^{\alpha\beta}(q)=\langle z^{\alpha},z^{\beta}\rangle^*$.
\item For $a\in \mathcal{Z}\setminus\{0\}$ satisfying
$a\cdot d=d_{\alpha}+d_{\beta}-d_1-2$,
$$
\frac{\partial^a \tilde{\omega}^{\alpha\beta}}{\partial  z^a}(q)=0~.
$$
\end{enumerate}
\end{lemma}
\begin{proof}
(1)  Since $\displaystyle{dx^{\alpha}=\sum_{\mu=1}^n\frac{\partial  x^{\alpha}}{\partial  z^{\mu}} dz^{\mu}}$,  $\tilde{\omega}$ is written as follows.
\begin{equation}\nonumber
\tilde{\omega}^{\alpha\beta}=\sum_{\mu,\nu=1}^n
\frac{\partial x^{\alpha}}{\partial z^{\mu}}
\frac{\partial x^{\beta}}{\partial z^{\nu}}
\langle dz^{\mu},dz^{\nu}\rangle^*
=
\sum_{\mu,\nu=1}^n J^{\alpha}_{\mu} J^{\beta}_{\nu}\langle z^{\mu},z^{\nu}\rangle^*
~.
\end{equation}
The  compatibility of $x$ with $z$  implies that $J^{\alpha}_{\mu}(q)=\delta^{\alpha}_{\mu}$.  Therefore 
evaluating at  $q$,  we obtain  $\tilde{\omega}^{\alpha\beta}(q)=\langle z^{\alpha},z^{\beta}\rangle^*$.
\\
(2) Take 
 $a\in \mathcal{Z}\setminus\{0\}$ satisfying
$a\cdot d=d_{\alpha}+d_{\beta}-d_1-2$. Then
\begin{equation}\nonumber
\begin{split}
\frac{\partial^a \tilde{\omega}^{\alpha\beta}}{\partial  z^a}(q)&=
\sum_{\mu,\nu=1}^n
\sum_{\begin{subarray}{c}a=a'+a''\\a',a''\neq 0\end{subarray}}
\frac{a!}{a'! a''!}
\underbrace{\frac{\partial^{a'}J^{\alpha}_{\mu}}{\partial  z^{a'}}(q)
\frac{\partial^{a''}J^{\alpha}_{\mu}}{\partial  z^{a''}}(q)}_{(\star)}
\langle z^{\mu},z^{\nu}\rangle^*
\\
&+\sum_{\nu=1}^n \frac{\partial^a J^{\beta}_{\nu}}{\partial z^a}(q)
\langle z^{\alpha},z^{\nu}\rangle^*
+\sum_{\mu=1}^n \frac{\partial^a J^{\alpha}_{\mu}}{\partial z^a}(q)
\langle z^{\mu},z^{\beta}\rangle^*
\end{split}
\end{equation}
We show that each term in the RHS is zero.
Start with the first term.
If $(\star)\neq 0$,  then by Lemma \ref{lemma2} (2), 
$$
\begin{cases}
a'\cdot d+d_{\mu}=d_{\alpha}+k_1 d_1\quad (k_1\geq 1)\\
a''\cdot d+d_{\nu}=d_{\beta}+k_2d_1\quad(k_2\geq 1)
\end{cases}~.
$$
Adding these, we have
\begin{equation}\nonumber
\begin{split}
a\cdot d&=d_{\alpha}+d_{\beta}-d_{\mu}-d_{\nu}+(k_1+k_2)d_1
\\&=d_{\alpha}+d_{\beta}-2+(k_1+k_2-1)d_1>d_{\alpha}+d_{\beta}-2-d_1.
\end{split}
\end{equation}
In the second equality,  we replaced $d_{\mu}+d_{\nu}$ by $d_1+2$.
This is allowed because $\langle z^{\mu},z^{\nu}\rangle^*$
vanishes if $d_{\mu}+d_{\nu}\neq  d_1+2$.
Therefore,  the assumption
$a\cdot d=d_{\alpha}+d_{\beta}-2-d_1$ implies that $(\star)=0$.
In the similar fashion,  we can show that the other terms vanish.
\end{proof}

\begin{lemma}
\begin{enumerate}
\item For $1\leq \alpha,\beta\leq n$,
$$
\tilde{\eta}^{\alpha\beta}=\frac{d_1}{z^1(q)}\langle z^{\alpha},z^{\beta}\rangle^*~.
$$
\item 
If $d_{\alpha}+d_{\beta}\neq d_1+2$, $\tilde{\eta}^{\alpha\beta}=\eta_{\alpha\beta}=0$.
\end{enumerate}
\end{lemma}
\begin{proof}
(1)
Given that $\tilde{\omega}^{\alpha\beta}$ is  a $G$-invariant homogeneous polynomial of degree
$d_{\alpha}+d_{\beta}-2$,
$\tilde{\omega}^{\alpha\beta}$ is written in the following form:
\begin{equation}\label{tilde-omega2}
\tilde{\omega}^{\alpha\beta}
=
\tilde{\eta}^{\alpha\beta} x^1+
\sum_{\begin{subarray}{c}c\in \mathcal{Z};\\
c\cdot d=d_{\alpha}+d_{\beta}-2\end{subarray}}
B'_c x^c~,
\quad
\tilde{\eta}^{\alpha\beta}=\sum_{\begin{subarray}{c}b\in \mathcal{Z};\\
b\cdot d=d_{\alpha}+d_{\beta}-2-d_1\end{subarray}}
B_b x^b ~.
\end{equation}
Here $B_b,B'_c$'s are some constants.
Therefore we only have to show that
$$
B_a=\begin{cases}
\frac{d_1}{z^1(q)}
\langle z^{\alpha},z^{\beta}\rangle^*&(a=0)\\
0&(a\in  \mathcal{Z}\setminus\{0\}, a\cdot d=d_{\alpha}+d_{\beta}-2-d_1)
\end{cases}~~.
$$

To compute $B_0$,
we evaluate  
\eqref{tilde-omega2} at $q$, compare with Lemma \ref{tilde-omega-q} (1)
and  obtain
$$
\tilde{\omega}^{\alpha\beta}(q)
=B_0 x^1(q)=\langle  z^{\alpha},z^{\beta}\rangle^*~.$$
Since $x^1(q)=\frac{z^1(q)}{d_1}$ by Lemma  \ref{compatible-x-form},
$B_0=\frac{d_1}{z^{1}(q)}\langle  z^{\alpha},z^{\beta}\rangle^*$.

Next take $a\in \mathcal{Z}\setminus\{0\}$ satisfying $a\cdot d=d_{\alpha}+d_{\beta}-2-d_1$.
Differentiating \eqref{tilde-omega2} by $z^a$
and evaluating at $q$, 
we obtain
\begin{equation}\nonumber
\begin{split}
0=\frac{\partial^a \tilde{\omega}^{\alpha\beta}}{\partial z^a}(q)
&=\sum_{\begin{subarray}{c}b\in \mathcal{Z};\\
b\cdot d=d_{\alpha}+d_{\beta}-2-d_1\end{subarray}}
B_b 
\Bigg(
\frac{\partial x^b}{\partial z^a}(q) x^1(q)
\\&+
\sum_{\begin{subarray}{c}a'+a''=a\\
|a''|=1\end{subarray}}
\frac{a!}{a'!a''!}
\frac{\partial^{a'} x^b}{\partial z^{a'}} (q)
\frac{\partial^{a''}  x^1}{\partial z^{a''}}(q)
+
\sum_{\begin{subarray}{c}a'+a''=a\\
|a''|\geq 2\end{subarray}}
\frac{a!}{a'!a''!}
{\frac{\partial^{a'} x^b}{\partial z^{a'}} (q)
\frac{\partial^{a''}  x^1}{\partial z^{a''}}(q)}
\Bigg)
\\&+
\sum_{\begin{subarray}{c}c\in \mathcal{Z};\\
c\cdot d=d_{\alpha}+d_{\beta}-2\end{subarray}}
B'_c 
{\frac{\partial^a x^c}{\partial z^a}}(q)~.
\end{split}
\end{equation}
In the RHS,
the first term does not vanish only if $b=a$,
and 
the other terms vanish by Lemma \ref{a-b-improved}.
Comparing with Lemma \ref{tilde-omega-q} (2), 
we have
$$
\frac{\partial^a \tilde{\omega}^{\alpha\beta}}{\partial z^a}(q)
=
B_a \frac{\partial  z^a}{\partial  z^a}(q) x^1(q)
=a! B_a x^1(q)=0~.
$$
Thus  $B_a=0$. 
\\
(2) follows from (1) and Lemma \ref{metric1}.
\end{proof}

%%%%%%%%%%%%%%%%%%
\subsection{Potential function}
%%%%%%%%%%%%%%%%%%
Let us set
\begin{equation}\label{def-hatx}
\hat{x}_{\alpha}=\sum_{\mu=1}^n x^{\mu}\eta_{\mu\alpha}
\end{equation}
\begin{lemma}\label{supple}
Let $1\leq \alpha,\beta,\gamma\leq n$.
For $a\in \mathcal{Z}$ satisfying $a\cdot d=2d_1+2-d_{\alpha}-d_{\beta}-d_{\gamma}$,
the following holds.
$$
\frac{1}{d_1+1-d_{\gamma}}\frac{\partial^a}{\partial z^a}\frac{\partial^2 \hat{x}_{\gamma}}{\partial  z^{\alpha}\partial z^{\beta}}(q)
=(\beta\leftrightarrow \gamma)~.
$$
\end{lemma}
\begin{proof}
By the expression of $C_{\alpha\beta}^{\gamma}$ in Theorem \ref{main-theorem2},
 we have
$$ \sum_{\lambda=1}^n C_{\alpha\beta}^{\lambda}\eta_{\lambda\gamma}
 =
 \sum_{\lambda=1}^n
 \frac{z^1(q)}{d_{\lambda}-1}\sum_{\begin{subarray}{c}
 a\in \mathcal{Z};\\
 a\cdot d=d_1+d_{\lambda}-d_{\alpha}-d_{\beta}
 \end{subarray}}
 \frac{\partial^a}{\partial z^a}\left(
 \frac{\partial^2 x^{\lambda}}{\partial z^{\alpha}\partial z^{\beta}}(q)
 \frac{x^a}{a!}
 \right) \eta_{\lambda\gamma}~.
 $$
 In the RHS,
 $d_{\lambda}$ can be replaced by $d_1+2-d_{\gamma}$
 because  of  the factor $\eta_{\lambda\gamma}$.
Therefore
 $$\sum_{\lambda=1}^n C_{\alpha\beta}^{\lambda}\eta_{\lambda\gamma}
=
 \frac{z^1(q)}{d_1+1-d_{\gamma}}
 \sum_{\begin{subarray}{c}
 a\in \mathcal{Z};\\
 a\cdot d=d_1+d_{\lambda}-d_{\alpha}-d_{\beta}
 \end{subarray}}
 \frac{\partial^a}{\partial z^a}\frac{\partial^2 \hat{x}_{\gamma}}{\partial  z^{\alpha}\partial z^{\beta}}
 (q) 
\frac{x^a}{a!}~.
$$ 
Then the statement follows from  \eqref{C-symmetry}.
 \end{proof}
 
 Let us put
$$
I_{\text{cox}}=\{a\in \mathbb{Z}_{\geq 0}^n \mid a\cdot d=2d_1+2, |a|\geq 3\}~,
\quad 
l(a)=\#\{\alpha\mid a_{\alpha}\neq 0\} \quad (a\in \mathbb{Z}_{\geq 0}^n)~,
$$
and
\begin{equation}\label{potential-F}
\mathcal{F}=\sum_{\lambda=1}^n \frac{z^1(q)}{d_1+1-d_{\lambda}}
\sum_{\begin{subarray}{c}b\in I_{\rm{cox}};\\
b_{\lambda}\geq 1
 \end{subarray}}
 \frac{\partial^{b-\bm{e}_{\lambda}}\hat{x}_{\lambda}}{\partial z^{b-\bm{e}_{\lambda}}}
 (q)\cdot \frac{x^b}{b! l(b)}~.
\end{equation}
\begin{theorem}\label{thm:potential}
We have
$$
\frac{\partial \mathcal{F}}{\partial x^{\gamma}}=\sum_{\lambda=1}^n
\mathcal{G}^{\lambda}\eta_{\lambda\gamma}~.
$$
In other words,  we have
$$
\frac{\partial^3 \mathcal{F}}{\partial x^{\alpha}\partial x^{\beta}\partial x^{\gamma}}
=\sum_{\lambda=1}^n C_{\alpha\beta}^{\lambda}\eta_{\lambda\gamma}~.
$$
\end{theorem}
\begin{proof} 
Differentiating \eqref{potential-F} by $x^{\gamma}$, we have
\begin{equation}\nonumber
\begin{split}
\frac{\partial \mathcal{F}}{\partial x^{\gamma}}
&=
\sum_{\lambda=1}^n \frac{z^1(q)}{d_1+1-d_{\lambda}}
\sum_{\begin{subarray}{c}b\in I_{\rm{cox}};\\
b_{\lambda}\geq 1,\\b_{\gamma}\geq 1
 \end{subarray}}
 \underbrace{
 \frac{\partial^{b-\bm{e}_{\lambda}}\hat{x}_{\lambda}}{\partial z^{b-\bm{e}_{\lambda}}}
 (q)
 }_{(\star)}\cdot \frac{b_{\gamma} x^{b-\bm{e}_{\gamma}}}{b! l(b)}
\end{split}
\end{equation}
When $\lambda\neq \gamma$, 
$b\in I_{\rm{cox}}$ appearing in $(\star)$
satisfies  $b_{\lambda},b_{\gamma}\geq 1$ and $|b|\geq 3$.
So if we put
$b':=b-\bm{e}_{\gamma}-\bm{e}_{\lambda}$,
$b'$ belongs to $\mathbb{Z}_{\geq 0}^n\setminus \{0\}$.
Moreover,  $b'_1=0,1$ since
 $$b'_1\cdot d_1\leq  b'\cdot d=b\cdot d-d_{\gamma}-d_{\lambda}\leq 2d_1-2<2d_1.$$
Then we can take some $1\leq \alpha\leq n$ such that
$b''=b'-\bm{e}_{\alpha}\in \mathcal{Z}$.
Indeed, there exists $2\leq \alpha\leq n$ such that $b'_{\alpha}>0$ when $b'_1=0$
and $\alpha=1$
when $b'_1=1$.
Therefore we have
$$(\star)=
\frac{\partial^{b''+\bm{e}_{\alpha}+\bm{e}_{\gamma}}\hat{x}_{\lambda}}
{\partial z^{b''+\bm{e}_{\alpha}+\bm{e}_{\gamma}}}(q)
=\frac{\partial^2}{\partial z^{\alpha}\partial z^{\gamma}}
\frac{\partial^{b''} \hat{x}_{\lambda}}{\partial z^{b''}}(q)~.
$$
Since $b''\cdot d=2d_1+2-d_{\gamma}-d_{\lambda}-d_{\alpha}$,
we can apply Lemma \ref{supple} and  we obtain
$$
(\star)=\frac{d_1+1-d_{\lambda}}{d_1+1-d_{\gamma}}\cdot
\frac{\partial^2}{\partial z^{\alpha}\partial z^{\lambda}}
\frac{\partial^{b''} \hat{x}_{\gamma}}{\partial z^{b''}}(q)
=
\frac{d_1+1-d_{\lambda}}{d_1+1-d_{\gamma}}\cdot
 \frac{\partial^{b-\bm{e}_{\gamma}}\hat{x}_{\gamma}}{\partial z^{b-\bm{e}_{\gamma}}}
 (q).
$$
Thus
\begin{equation}\nonumber
\begin{split}
\frac{\partial \mathcal{F}}{\partial x^{\gamma}}
&=\frac{z^1(q)}{d_1+1-d_{\gamma}}
\sum_{\lambda=1}^n 
\sum_{\begin{subarray}{c}b\in I_{\rm{cox}};\\
b_{\lambda}\geq 1,\\b_{\gamma}\geq 1
 \end{subarray}}
 \frac{\partial^{b-\bm{e}_{\gamma}}\hat{x}_{\gamma}}{\partial z^{b-\bm{e}_{\gamma}}}
 (q)
 \cdot \frac{b_{\gamma} x^{b-\bm{e}_{\gamma}}}{b! l(b)}
 \\
 &=\frac{z^1(q)}{d_1+1-d_{\gamma}}
\sum_{\begin{subarray}{c}b\in I_{\rm{cox}};\\b_{\gamma}\geq 1
 \end{subarray}}
 \underbrace{\#\{\lambda\mid b_{\lambda}\geq 1\}}_{=l(b)}\cdot
 \frac{\partial^{b-\bm{e}_{\gamma}}\hat{x}_{\gamma}}{\partial z^{b-\bm{e}_{\gamma}}}
 (q)
 \cdot \frac{b_{\gamma} x^{b-\bm{e}_{\gamma}}}{b! l(b)}
 \\
 &=\frac{z^1(q)}{d_1+1-d_{\gamma}}
\sum_{\begin{subarray}{c}b\in I_{\rm{cox}};\\b_{\gamma}\geq 1
 \end{subarray}}
 \frac{\partial^{b-\bm{e}_{\gamma}}\hat{x}_{\gamma}}{\partial z^{b-\bm{e}_{\gamma}}}
 (q)
 \cdot \frac{b_{\gamma} x^{b-\bm{e}_{\gamma}}}{b!}~.
\end{split}
\end{equation}
If we write $a=b-\bm{e}_{\gamma}$,
then 
$a\cdot d=2d_1+2-d_{\gamma}$ and $|a|\geq 2$, which implies  that
$a\in I_{n+1-\gamma}^{(1)}$.
Thus we obtain
$$
\frac{\partial \mathcal{F}}{\partial x^{\gamma}}
=\frac{z^1(q)}{d_1+1-d_{\gamma}}
\sum_{\begin{subarray}{c}a\in I_{n+1-\gamma}^{(1)}
 \end{subarray}}
 \frac{\partial^a \hat{x}_{\gamma}}{\partial z^a}
 (q)
 \cdot \frac{x^a}{a!}~.$$
 
 On the other  hand,
 \begin{equation}\nonumber
 \begin{split}
 \sum_{\lambda=1}^n \mathcal{G}^{\lambda}\eta_{\lambda\gamma}
 &=
 \sum_{\lambda=1}^n \frac{z^1(q)}{d_{\lambda}-1}
 \sum_{a\in I_{\lambda}^{(1)}}  \frac{\partial^a x^{\lambda}}{\partial  z^a}(q)
 \frac{x^a}{a!}\cdot \eta_{\lambda\gamma}~.
 \end{split}
 \end{equation}
Because of  the factor $\eta_{\lambda\gamma}$,
we can replace $d_{\lambda}$ with $d_1+2-d_{\gamma}$.
Moreover $I_{\lambda}^{(1)}=I_{n+1-\gamma}^{(1)}$  since
$d_{n+1-\gamma}=d_1+2-d_{\gamma}$.
Thus we obtain the statement.
\end{proof}

%%%%%%%%%%%%%%%%%%%%%%%%%%%%%%%%%%%%%%
\subsection{Satake's expression for $C_{\alpha\beta}^{\gamma}$}
%%%%%%%%%%%%%%%%%%%%%%%%%%%%%%%%%%%%%%

\begin{proposition}\label{C-Satake}
When $G$ is an irreducible finite Coxeter group, 
$C_{\alpha\beta}^{\gamma}$ in Theorem \ref{main-theorem2} can be
written as follows.
\begin{equation}\nonumber
C_{\alpha\beta}^{\gamma}
=\frac{z^1(q)}{d_1+d_{\gamma}-d_{\beta}}
\sum_{\begin{subarray}{c}a\in \mathcal{Z};
\\a\cdot d=d_1+d_{\gamma}-d_{\alpha}-d_{\beta}\end{subarray}}
\frac{\partial^a}{\partial z^a}
\frac{\partial}{\partial z^{\alpha}}\left(
\frac{\partial x^{\gamma}}{\partial z^{\beta}}
+\sum_{\lambda=1}^n \tilde{\eta}^{\lambda\gamma}\frac{\partial \hat{x}_{\beta}}{\partial z^{\lambda}}
\right)(q)
\cdot \frac{x^a}{a!}~.
\end{equation} 
See \eqref{def-hatx} for the definition of  $\hat{x}_{\beta}$.
\end{proposition}

\begin{proof}
By Theorem \ref{main-theorem2},  we have
\begin{equation}\label{C2}
\begin{split}
\sum_{\lambda,\mu=1}^n C_{\alpha\lambda}^{\mu}
\,\eta_{\mu\beta}\,
\tilde{\eta}^{\lambda\gamma}
&=
\sum_{\lambda,\mu=1} \frac{z^1(q)}{d_{\mu}-1}
\frac{\partial^a}{\partial z^a}\left(
\frac{\partial^2  x^{\mu}}{\partial z^{\alpha}\partial z^{\lambda}}
\right)(q) \frac{x^a}{a!}\cdot \eta_{\mu\beta}\cdot \tilde{\eta}^{\lambda\gamma}
\\
&=\frac{z^1(q)}{d_1+1-d_{\beta}}
\frac{\partial^a}{\partial z^a}\frac{\partial}{\partial z^{\alpha}}
\left(\sum_{\lambda=1}^n \tilde{\eta}^{\lambda\gamma}
\frac{\partial  \hat{x}_{\beta}}{\partial z^{\lambda}}
\right)(q) \frac{x^a}{a!}~.
\end{split}
\end{equation}
In passing to the  second line,
we replaced  $d_{\mu}$ by $d_1+2-d_{\beta}$.
This is allowed because of the factor $\eta_{\mu\beta}$.

By \eqref{C-symmetry},  it holds that
$$
C_{\alpha\beta}^{\gamma}=\sum_{\lambda,\mu=1}^n C_{\alpha\lambda}^{\mu}
\,\eta_{\mu\beta}\,
\tilde{\eta}^{\lambda\gamma}~.
$$
Therefore
\begin{equation}\nonumber 
\begin{split}
(d_{\gamma}-1)C_{\alpha\beta}^{\gamma}
+(d_1+1-d_{\beta})\sum_{\lambda,\mu=1}^n C_{\alpha\lambda}^{\mu}
\,\eta_{\mu\beta}\,
\tilde{\eta}^{\lambda\gamma}
&=(d_1+d_{\gamma}-d_{\beta})C_{\alpha\beta}^{\gamma}~.
\end{split}
\end{equation}
Then the statement follows from Theorem  \ref{main-theorem2} and \eqref{C2}.
 \end{proof}

\begin{remark}
The expression in Proposition \ref{C-Satake}
agrees with Satake's expression \cite[Theorem 6.9]{Satake2020}.
\end{remark}

%%%%%%%%%%%%%%%%%%%%%%%
\section{Exceptional groups of rank $2$}
\label{sec:examples}
%%%%%%%%%%%%%%%%%%%%%%%
In this section, 
we give examples of admissible triplets and good basic invariants
for some exceptional groups of rank two.
The reader can find data for those groups, such as the degrees, 
generators and basic invariants in \cite[\S 6]{LehrerTaylor}.

Hereafter in this and the next sections, we change  superscripts such as 
$z^{\alpha},x^{\alpha}$
to subscripts  $z_{\alpha},x_{\alpha}$ to save space.

%%%%%%%%%%%
\subsection{$G_5$}
%%%%%%%%%%%
The duality group $G_5$ is generated by $r_1,r_2'$ where
\begin{equation}\nonumber
r_1=\frac{\omega}{2}\begin{bmatrix}
-1-i&1-i\\
-1-i&-1+i
\end{bmatrix}~,\quad
r_2'=\frac{\omega}{2}
\begin{bmatrix}
-1+i&1-i\\
-1-i&-1-i
\end{bmatrix}
\quad (\omega=e^{\frac{2\pi i}{3}})~.
\end{equation}
The degrees of $G_5$ are
\begin{equation}\nonumber
d_1=12,~~d_2=6~,
\end{equation}
and a set of basic invariants is 
$(f_{T}^3, t_T)$
where 
\begin{equation}\label{fT-tT}
f_T=u_1^4+2i\sqrt{3}u_1^2u_2^2+u_2^4,\quad
t_T=u_1^5 u_2-u_1 u_2^5~.
\end{equation}
Here $u_1,u_2$ are the standard coordinates of $V=\mathbb{C}^2$.

We take $g=(r_2'r_1)^{-1}$ as a $\zeta$-regular element with
$\zeta=e^{\frac{\pi i}{6}}$.
 Eigenvalues  of $g$ are  $e^{\frac{\pi i}{6}}$, $e^{\frac{7\pi i}{6}}$ and
corresponding eigenvectors are
$$
q_1=\frac{1}{\sqrt{2}}\begin{bmatrix}1\\1
\end{bmatrix}~,
\quad 
q_2=\frac{1}{\sqrt{2}}\begin{bmatrix}
-1\\1
\end{bmatrix}~.
$$
By substituting $q_1$ into the Jacobian, we see that
$$
\det\begin{bmatrix}
\frac{\partial x_1}{\partial u_1}&\frac{\partial x_1}{\partial u_2}\\
\frac{\partial x_2}{\partial u_1}&\frac{\partial x_2}{\partial u_2}
\end{bmatrix}(q_1)=12\neq 0~,
$$
which implies that $q_1$ is a regular vector.

We take $(g,\zeta,q_1)=((r_2'r_1)^{-1},e^{\frac{\pi i}{6}},q_1 )$ as an admissible triplet.
If we denote the coordinates associated to the basis $q_1,q_2$ by $z_1,z_2$, then
$z=(z_1,z_2)$ is a $(g,\zeta)$-graded coordinate system.
Now we find a  set of good basic invariants using the method in the proof of Theorem \ref{existence-good}.
Substituting the relations
\begin{equation}\nonumber
\begin{split}
u_1=\frac{z_1-z_2}{\sqrt{2}}~,
\quad
u_2=\frac{z_1+z_2}{\sqrt{2}}~,
\end{split}
\end{equation}
into $f_T^3$ and $t_T$, we obtain
\begin{equation}\nonumber
\begin{split}
f_T^3&=-z_1^{12}+6i\sqrt{3}z_1^{10}z_2^2+
33z_1^8z_2^4-12i \sqrt{3} z_1^6z_2^6
\\&+33z_1^4z_2^8+6i\sqrt{3} z_1^2 z_2^{10}-z_2^{12}~,
\\
t_T&=-z_1^5 z_2+z_1 z_2^5~.
\end{split}
\end{equation}
It is easy to check that $(-\frac{f_T^3}{12},-t_T)$ is compatible with $z$ at $q_1$.

Let  us impose the goodness condition (Definition \ref{def-good}).
$I_{\alpha}^{(0)}$ ($\alpha=1,2$) defined in \eqref{defIalpha} are
$I_{1}=\{(0,2)\}$ and $I_2=\emptyset$.
We assume that
$$
x_1=-\frac{1}{12}(f_T^3+A \, t_T^2)~,
\quad 
x_2=-t_T
$$
are good basic invariants.
The goodness condition 
becomes
$$
\frac{\partial^2 x_1}{\partial z_2^2}(q_1)=
-\frac{1}{6}(6i\sqrt{3}+A)=0~.
$$
Thus we have $
A=-6i\sqrt{3}
$ and 
\begin{equation}\label{G5-good}
\begin{split}
x_1&=-\frac{1}{12}\left(f_T^3-6i\sqrt{3}\,t_T^2\right)
=\frac{1}{12}\left(
z_1^{12}-33z_1^8z_2^4-33z_1^4z_2^8+z_2^{12}
\right)~,
\\ 
x_2&=-t_T=z_1^5 z_2-z_1 z_2^5~.
\end{split}
\end{equation}
form a set of  good basic invariants compatible with $z$ at $q_1$.
The potential vector field $\mathcal{G}^{\gamma}$ in Corollary \ref{vector-potential} 
is given by 
\begin{equation}\label{G5-product}
\mathcal{G}^1=\dfrac{x_1^2}{2}-\dfrac{x_2^4}{4},\quad
\mathcal{G}^2=x_1x_2.
\end{equation}
The good basic invariants \eqref{G5-good} and
$C_{\alpha\beta}^{\gamma}
$'s  
obtained from \eqref{G5-product} 
are in agreement with the results in \cite[Table C6]{KMS2018}.
%%%%%%%%%%%
\subsection{$G_6$}
%%%%%%%%%%%
The duality group $G_6$ is generated by $r,r_1$ where
\begin{equation}\nonumber
r=\begin{bmatrix}
1&0\\0&-1
\end{bmatrix}~,\quad
r_1=\frac{\omega}{2}\begin{bmatrix}
-1-i&1-i\\
-1-i&-1+i
\end{bmatrix}\quad
(\omega=e^{\frac{2\pi i}{3}})~.
\end{equation}
The degrees are
\begin{equation}\nonumber
d_1=12,~~d_2=4~,
\end{equation}
and a set of basic invariants is $(t_T^2,f_T)$,
where $f_T,t_T$ are those defined in \eqref{fT-tT}.

We take $g=(r_1r)^{-1}$
as a $\zeta$-regular element with $\zeta=e^{\frac{\pi i}{6}}$.
Eigenvalues of $g$ are  $e^{\frac{\pi i}{6}}$,$-i$
and corresponding eigenvectors are
$$
q_1=\begin{bmatrix}
\frac{(1-i)\sqrt{3+\sqrt{3}}}{2\sqrt{3}}\\
\frac{\sqrt{3-\sqrt{3}}}{\sqrt{6}}
\end{bmatrix}~,
\quad 
q_2=\begin{bmatrix}
\frac{(-1+i)\sqrt{3-\sqrt{3}}}{2\sqrt{3}}\\
\frac{\sqrt{3+\sqrt{3}}}{\sqrt{6}}
\end{bmatrix}~,
$$
and $q_1$ is regular.
We take $(g,\zeta,q)=((r_1r)^{-1}, e^{\frac{\pi i}{6}},q_1)$
as an admissible triplet.
The $(g,\zeta)$-graded coordinates $z_1,z_2$ 
associated to $q_1,q_2$
are related to 
the standard coordinates $u_1,u_2$ by
\begin{equation}\nonumber
\begin{split}
u_1&=\frac{(1-i)}{2\sqrt{3}}
\left(\sqrt{3+\sqrt{3}}\,z_1-\sqrt{3-\sqrt{3}} \,z_2\right)~,
\\
u_2&=\frac{1}{\sqrt{6}}
\left(\sqrt{3-\sqrt{3}}\,z_1+\sqrt{3+\sqrt{3}}\, z_2\right)~.
\end{split}
\end{equation}
Then the set of basic invariants which is good with respect to $((r_1r)^{-1}, e^{\frac{\pi i}{6}},q_1)$ and which is compatible with $z$ at $q_1$ is given by
\begin{equation}\label{G6-good}
\begin{split}
x_1&=\frac{9i}{8}\left(t_T^2+\frac{5i}{96\sqrt{3}}f_T^3\right)
=\frac{z_1^{12}}{12}+\frac{11}{4}z_1^6z_2^6-\frac{55}{24\sqrt{2}}z_1^3z_2^9
+\frac{z_2^{12}}{32},
\\
x_2&=\frac{\sqrt{6}}{8} f_T=z_1^3z_2+\frac{z_2^4}{2\sqrt{2}}~.
\end{split}
\end{equation}
The potential vector field $\mathcal{G}^{\gamma}$ in Corollary \ref{vector-potential} 
is given by 
\begin{equation} \label{G6-product}
\mathcal{G}^1=\dfrac{x_1^2}{2}+\dfrac{x_2^6}{4},\quad
\mathcal{G}^2=x_1x_2+\dfrac{x_2^4}{6\sqrt{2}}~.
\end{equation}
The good basic invariants \eqref{G6-good} and 
$C_{\alpha\beta}^{\gamma}
$'s  
obtained from \eqref{G6-product} 
are in agreement with the results in \cite[Table C6]{KMS2018}.

%%%%%%%%%%%%%%%%%%%
\subsection{$G_7,G_{11},G_{19}$}
%%%%%%%%%%%%%%%%%%%
The nonduality groups of rank $2$, $G_{7},G_{11}$ and $G_{19}$, have $\mathfrak{a}(d_1)=\mathfrak{b}(d_1)=2$.
For these groups, see Remark \ref{ad=2}.

%%%%%%%%%%%%%%%%%%%%%
\subsection{$G_{12}, G_{13},G_{22}$}
%%%%%%%%%%%%%%%%%%%%%
The nonduality groups of rank $2$, $G_{12},G_{13}$ and $G_{22}$, have $\mathfrak{a}(d_1)=\mathfrak{b}(d_1)=1$.
For any of these groups, $I_1=I_2=\emptyset$
since $d_2$ does not divide $d_1$.
Therefore the goodness condition is empty and hence
any set of basic invariants is good with respect to any admissible triplet.
(For basic invariants of these groups,  see \cite[\S 6.6]{LehrerTaylor}.)

We list an example of an admissible triplet $(g,\zeta,q)$ for each of these groups.
Let us set 
\begin{equation}\nonumber
\begin{split}
r_3&=\frac{1}{\sqrt{2}}\begin{bmatrix}1&-1\\-1&-1\end{bmatrix},\quad
r_3'=\frac{1}{\sqrt{2}}\begin{bmatrix}1&1\\1&-1\end{bmatrix}, \\
r_3''&=\frac{1}{\sqrt{2}}\begin{bmatrix}0&1+i\\1-i&0\end{bmatrix},\quad
r=\begin{bmatrix}1&0\\0&-1\end{bmatrix},\\
r'&=\begin{bmatrix}0&1\\1&0\end{bmatrix},\quad
r''=\begin{bmatrix}
\frac{1}{2}+\cos \frac{2\pi}{5}&\frac{1}{2}+i\cos \frac{2\pi}{5}\\
\frac{1}{2}-i\cos \frac{2\pi}{5}&-\frac{1}{2}-\cos \frac{2\pi}{5}
\end{bmatrix}~.
\end{split}
\end{equation}
\begin{itemize}
\item
For $G_{12}=\langle r_3,r_3',r_3''\rangle$, the degrees are $d_1=8,d_2=6$.
$$
g=r_3''r_3'r_3=
\frac{1}{\sqrt{2}}\begin{bmatrix}1+i&0\\0&-1+i\end{bmatrix},
\quad \zeta=e^{\frac{\pi i}{4}},\quad q=\begin{bmatrix}1\\0\end{bmatrix}~.
$$
\item
For $G_{13}=\langle r, r_3,r_3''\rangle $, the degrees are $d_1=12,d_2=8$.
$$
g=rr_3r_3''=\frac{1}{2}\begin{bmatrix}
-1+i&1+i\\1-i&1+i
\end{bmatrix},\quad \zeta=e^{\frac{\pi i}{6}},\quad
 q=\frac{1}{\sqrt{3-\sqrt{3}}}\begin{bmatrix}\frac{(1+i)(\sqrt{3}-1)}{2}\\1
\end{bmatrix}.
$$
\item
For $G_{22}=\langle r,r',r''\rangle $, the degrees are $d_1=20,d_2=12$.
\begin{equation}\begin{split}\nonumber
g&=r''r'r=\begin{bmatrix}
\frac{1}{2}+i\cos \frac{2\pi}{5}&-\frac{1}{2}-\cos \frac{2\pi}{5}\\
-\frac{1}{2}-\cos \frac{2\pi}{5}&-\frac{1}{2}+i\cos \frac{2\pi}{5}
\end{bmatrix}
,\quad 
\zeta=e^{\frac{\pi i}{10}},
\\
q&=
\frac{1}{\sqrt{5-\sqrt{5}+2\sqrt{5-2\sqrt{5}}}}
\begin{bmatrix}
-\frac{2+\sqrt{10-2\sqrt{5}}}{\sqrt{5}+1}\\1
\end{bmatrix}.
\end{split}
\end{equation}
\end{itemize}

\begin{remark}
Although the non-duality groups $G_{12},G_{13},G_{22}$ do not admit natural Saito structure, one can formally
define $\mathcal{G}^1,\mathcal{G}^2$ 
and $C_1,C_2$ by \eqref{eq:pvf} and  by the relation stated in Corollary \ref{vector-potential}.
Then for any of these groups, we have
$$
\mathcal{G}^1=\frac{x_1^2}{2},\quad \mathcal{G}^2=x_1x_2,\quad
C_1=\begin{bmatrix}1&0\\0&1\end{bmatrix},\quad 
C_2=\begin{bmatrix}0&0\\1&0\end{bmatrix}.
$$
These together with $\Gamma_1=\Gamma_2=O_2$
give a  trivial Saito structure.
\end{remark}

%%%%%%%%%%%%%%%%%%%%%%%%%%%%%%
\section{$G(m,m, n)$ and $G(m,1,n)$ for $m\geq 2$}
\label{sec:examples2}
%%%%%%%%%%%%%%%%%%%%%%%%%%%%%%
First, we list a set of good basic invariants of these groups for $n\leq 4$.

%%%%%%%%%%%%%%%%%%%%%%%%
\subsection{$G(m,m,n+1)$ for $m\geq 2$}
\label{sec:mmn}
%%%%%%%%%%%%%%%%%%%%%%%%
Let $r_{i}$ ($i=1,\ldots, n$) be the $(n+1)\times (n+1)$ matrix 
obtained from the identity matrix $I_{n+1}$ by exchanging its $i$-th 
and $(i+1)$-th rows. Let $\zeta_{nm}$ be a primitive $nm$-th root of unity
and let 
$$s:=
\left[
\begin{array}{c|c}
\begin{array}{cc}
0&\zeta_{nm}^{-n}\\
\zeta_{nm}^{n}&0
\end{array}
&0\\
\hline
0&I_{n-1}
\end{array}
\right].
$$
Then the group $G(m,m,n+1)$ is generated by $r_1,\ldots,r_{n},s$.
Let $e_i(u)$ be the elementary symmetric polynomial
of degree $i$ in $u=(u_1,\ldots, u_{n+1})$. We set 
$u^m:=\left( u_1^m,\ldots, u_{n+1}^m\right)$
and 
\begin{equation}\nonumber
\sigma_i:=\begin{cases}
e_{n+1-i}(u^m)\quad (1\leq i \leq n),\\
e_{n+1}(u)\qquad (i=n+1).
\end{cases}
\end{equation}
It is known that $\sigma_1,\ldots,\sigma_{n+1}$ form
a set of basic invariants of $G(m,m,n+1)$. 
The degrees of $G(m,m,n+1)$ are $m,2m,\ldots, nm$ and $n+1$. 
The highest degree is $nm$ since we assumed $m\geq 2$.
It is not so convenient 
to arrange the degrees in the descending order for $G(m,m,n+1)$.
So  we arrange the degrees in the following manner:
\begin{equation}\label{Gmmn-degree}
d_1=nm, ~~d_2=(n-1)m,~~\ldots, d_{n-1}=2m,~~d_n=m,~~d_{n+1}=n+1~.
\end{equation}

Let
$$
g:=(r_{n-1}\cdots r_1)(r_{n}\cdots r_1)s(r_2\cdots r_{n})
=
\left[
\begin{array}{c|c}
\begin{array}{cc}
0&I_{n-1}\\
\zeta_{nm}^{n}&0
\end{array}
&0\\
\hline
0&\zeta_{nm}^{-n}
\end{array}
\right].
$$
Then $g$ is a $d_1$-regular element. Its eigenvalues are
$\lambda_i=\zeta_{nm}^{1+(i-1)m}$ ($i=1,\ldots,n$)
and $\zeta_{nm}^{-n}$,  and eigenvectors are
$$
q_i=\frac{1}{\sqrt{n}} \begin{bmatrix}
1\\ \lambda_i\\ \lambda_i^2\\ \vdots \\ \lambda_i^{n-1}\\0
\end{bmatrix}\quad
(1\leq i \leq n),\qquad
q_{n+1}=\begin{bmatrix}0\\ \vdots\\0 \\1\end{bmatrix}.
$$
Then we have an admissible triplet $(g, \zeta_{nm}, q_1)$.
Let $z=(z_1,\ldots, z_{n+1})$ be the dual coordinates of 
$q_1,\ldots, q_{n+1}$. Below, we list 
the set of good basic invariants compatible with $z$ at $q_1$.  
Notice that $z_1,\ldots, z_n$ are linear combinations of $u_1,\ldots,u_n$
and $z_{n+1}=u_{n+1}$. 
Therefore
\begin{equation}\label{sigma}
\begin{cases}
\sigma_i=e_{n+1-i}(u_1^m,\ldots, u_n^m) +z_{n+1}^m e_{n-i} (u_1^m,\ldots, u_n^m)
&(1\leq i\leq n)\\
\sigma_{n+1}=u_1\cdots u_n\cdot z_{n+1}
\end{cases}
~.
\end{equation}
Given that $z_{n+1}(q_1)=0$, we have  the following remark,
which is helpful in the computations in examples.
\begin{remark}
\begin{enumerate}
\item
For $1\leq i\leq n$ and $a\in \mathbb{Z}_{\geq 0}^{n+1}$ with $a_{n+1}=0$,
we have
\begin{equation}\nonumber
\frac{\partial^a \sigma_i}{\partial z^a} (q_1)
=\left(\frac{\partial^a}{\partial z^a} e_{n+1-i}(u_1^m,\ldots, u_n^m)\right) (q_1)~.
\end{equation}
\item For $1\leq i\leq n$ and 
$a\in \mathbb{Z}_{\geq 0}^{n+1}$ such that $a_{n+1}\neq m$, we have
\begin{equation}\nonumber
\frac{\partial^a \sigma_i}{\partial z^a} (q_1)=0~.
\end{equation}
\end{enumerate}
\end{remark}

%%%%%%%%%%%%%%%%
\subsubsection{$G(m,m,2)$}
%%%%%%%%%%%%%%%%
A set of good basic invariants is given by
$$x_1=\frac{1}{m}\sigma_1,\quad x_2=\sigma_2.$$
The potential vector field
$\mathcal{G}^{\gamma}$ 
is given by 
$$\mathcal{G}^1=\dfrac{x_1^2}{2}+\dfrac{x_2^m}{m(m-1)},\quad
\mathcal{G}^2=x_1x_2.$$

The group $G(m,m, 2)$ is nothing but the Coxeter group $I_2(m)$. 
The representation matrix of the invariant metric $\langle ~,~\rangle$ in  $q_1,q_2$ is given by 
$$
\kappa \begin{bmatrix}0&1\\1&0\end{bmatrix}
$$
where $\kappa$ is a nonzero constant.
Therefore  if we take $\kappa=d_1=m$, then 
$$
\eta=\begin{bmatrix}0&1\\1&0\end{bmatrix}
$$
and 
we have the potential function
\begin{equation}\label{I2m}
\mathcal{F}=\frac{1}{2}x_1^2x_2+\frac{1}{m(m-1)(m+1)}x_2^{m+1}.
\end{equation}

%%%%%%%%%%%%%%%%
\subsubsection{$G(m,m,3)$}
%%%%%%%%%%%%%%%%
A set of good basic invariants is given by
$$\begin{cases}
x_1=-\frac{(\sqrt{2})^{2m}}{2m}\left(
\sigma_1-\frac{1}{4m}\sigma_2^2\right),\\
x_2=\frac{(\sqrt{2})^m}{2m}\sigma_2,\\
x_3=\frac{(\sqrt{2})^2}{\zeta_{2m}}\sigma_3
,
\end{cases}
$$
where $\zeta_{2m}$ is a primitive $2m$-th root of unity.
The potential vector field $\mathcal{G}^{\gamma}$ is given by
\begin{equation}\nonumber
\begin{split}
\mathcal{G}^1&=\frac{x_1^2}{2}+\frac{m-1}{12}x_2^4-
\frac{(\sqrt{2})^m}{2m}x_2 x_3^m,
\\
\mathcal{G}^2&=x_1x_2+\frac{m-2}{6}x_2^3+\frac{(\sqrt{2})^m}{2m(m-1)}x_3^m,
\\
\mathcal{G}^3&=x_1x_3-\frac{1}{2}x_2^2x_3.
\end{split}
\end{equation}
These agree with the result presented in \cite{Kato2017}.

For the case with $m=2$, $G(2,2,3)$ is nothing but the Coxeter group 
$A_3$. 
The representation matrix of the invariant metric $\langle~,~\rangle$ in the standard basis 
is a nonzero-constant multiple of the identity matrix.
If we  take  the constant to $d_1=4$,  we have
$$
\eta=\begin{bmatrix}0&1&0\\1&0&0\\0&0&1\end{bmatrix}~.
$$
The reason that this is not anti-diagonal is our convention \eqref{Gmmn-degree};
the degrees of  $x_1,x_2,x_3$ are $d_1=4,d_2=2,d_3=3$.
We have the potential function
$$
\mathcal{F}=\frac{1}{2}x_1^2x_2+\frac{1}{2}x_1x_3^2-\frac{1}{4}x_2^2x_3^2
+\frac{1}{60}x_2^5.
$$
If we interchange $x_2$ and $x_3$ 
and take the normalization of $x$ into account,
this agrees with eq.(2.46) in \cite{Dubrovin1998}.

%%%%%%%%%%%%%%%%
\subsubsection{$G(m,m,4)$} 
%%%%%%%%%%%%%%%%
A set of good basic invariants is given by
$$
\begin{cases}
x_1=\frac{(\sqrt{3})^{3m}}{3m}\left(
\sigma_1-\frac{1}{3m}\sigma_2\sigma_3+\frac{3m+1}{54m^2}\sigma_3^3\right),~\\
x_2=-\frac{(\sqrt{3})^{2m}}{3m}\left(\sigma_2-\frac{m+1}{6m}\sigma_3^2 \right),\\
x_3=\frac{(\sqrt{3})^m}{3m} \sigma_3,~\\
x_4=\frac{(\sqrt{3})^3}{\zeta^3_{3m}}\sigma_4,
\end{cases}$$
where $\zeta_{3m}$ is a primitive $3m$-th root of unity.
The potential vector field $\mathcal{G}^{\gamma}$ is given by
\begin{equation}\nonumber
\begin{split}
\mathcal{G}^1&=\frac{x_1^2}{2}+\frac{x_2^3}{6}+\frac{(m-1)}{2}x_2^2 x_3^2
+\frac{(m-1)(m-2)}{8}x_2 x_3^4+\frac{(m-1)(m^2-2m+2)}{60}x_3^6\\
&+\frac{(\sqrt{3})^m}{6}x_3^2x_4^m-\frac{(\sqrt{3})^m}{3m}x_2x_4^m,\\
\mathcal{G}^2&=x_1x_2+\frac{m(m-1)}{6}x_2 x_3^3+\frac{m-1}{2}x_2^2x_3
+\frac{(m-1)(m-2)(m+1)}{40}x_3^5-\frac{(\sqrt{3})^m}{3m}x_3 x_4^m,
\\
\mathcal{G}^3&=x_1x_3+\frac{x_2^2}{2}+\frac{m-2}{2}x_2x_3^2
+\frac{(m-2)(m-3)}{24}x_3^4+\frac{(\sqrt{3})^m}{3m(m-1)}x_4^m,
\\
\mathcal{G}^4&=x_1x_4-x_2x_3x_4+\frac{1}{3}x_3^3x_4.
\end{split}
\end{equation}

For the case with $m=2$, we have $G(2,2,4)=D_4$.
The representation matrix of the invariant metric $\langle~,~\rangle$ in the standard basis 
is a nonzero-constant multiple of the identity matrix.
If we  set  the constant to $d_1=6$,  we have
$$
\eta=\begin{bmatrix}0&0&1&0\\0&1&0&0\\1&0&0&0\\0&0&0&1\end{bmatrix}
$$
and  we have the potential function
$$\mathcal{F}=
\frac{1}{2}x_1^2x_3+\frac{1}{2}x_1x_2^2+\frac{1}{2}x_1x_4^2
+\frac{1}{6}x_2^2x_3^3
+\frac{1}{6}x_2^3x_3+\frac{1}{210}x_3^7
-\frac{1}{2}x_2x_3x_4^2
+\frac{1}{6}x_3^3 x_4^2~.
$$
%%%%%%%%%%%%%%%%%%%%%%%
\subsection{$G(m,1,n)$ for $m\geq 2$}
\label{sec:m1n}
%%%%%%%%%%%%%%%%%%%%%%%
Let $\bar{r}_{i}$ ($i=1,\ldots, n-1$) be the $n\times n$ matrix 
obtained from the identity matrix $I_n$ by exchanging its $i$-th 
and $(i+1)$-th rows. Let $\zeta_{nm}$ be a primitive $nm$-th root of unity
and let $\bar{t}:=\mathrm{diag}[\zeta_{nm}^n,1,\ldots,1]$.
Then the group $G(m,1,n)$ is generated by $\bar{r}_1,\ldots, \bar{r}_{n-1},\bar{t}$.
Let $$
\bar{\sigma}_i:=e_{n+1-i}(u_1^m,\ldots, u_n^m)\quad  (i=1,\ldots,n). $$
Here $e_{i}$ denotes the $i$-th elementary symmetric polynomial.
Then $\bar{\sigma}_1,\ldots,\bar{\sigma}_n$
form a set of basic invariants of $G(m,1,n)$. Their degrees are given by 
$d_i=\deg \bar{\sigma}_i=(n+1-i)m$.
Let 
$$
\bar{g}:=\bar{r}_{n-1}\cdots \bar{r}_1\bar{t}=
\left[
\begin{array}{c|c}
0&I_{n-1}\\ \hline
\zeta_{nm}^n&0
\end{array}
\right].
$$
Then $\bar{g}$ is a $d_1$-regular element. Its eigenvalues are
$\lambda_i=\zeta_{nm}^{1+(i-1)m}$ ($i=1,\ldots,n$) 
and $\lambda_i$-eigenvectors are
$$
\bar{q}_i=\frac{1}{\sqrt{n}} \begin{bmatrix}
1\\ \lambda_i\\ \lambda_i^2\\ \vdots \\ \lambda_i^{n-1}
\end{bmatrix}.
$$
Then we have an admissible triplet $(\bar{g}, \zeta_{nm}, \bar{q}_1)$.
Let $z=(z_1,\ldots, z_n)$ be the dual coordinates of 
$\bar{q}_1,\ldots, \bar{q}_n$. Below, we list 
the set of good basic invariants compatible with $z$ at $\bar{q}_1$.  

When $m=2$, $G(2,1,n)$ is the finite Coxeter group $B_n$.
The representation matrix of the invariant metric 
$\langle ~,~\rangle$ with respect to the standard basis is a nonzero constant multiple of
the identity matrix. If we set the constant to $d_1=2n$,  we have
$$
\eta_{ij}
=\delta_{i+j,n+1}~.
$$

%%%%%%%%%%%%%%%%
\subsubsection{$G(m,1,2)$}
%%%%%%%%%%%%%%%%
A set of good basic invariants is given by
\begin{equation*}
\begin{cases}
x_1=\frac{(\sqrt{2})^{2m}}{-2m} 
\left(\bar{\sigma}_1-\frac{1}{4m}\bar{\sigma}_2^2 \right)\\
x_2=\frac{(\sqrt{2})^{m}}{2m} \bar{\sigma}_2
\end{cases}.
\end{equation*}
This agrees with the flat coordinates given in \cite[\S 7.5]{Arsie-Lorenzoni2016}, 
\cite{KMS2018}. The potential vector field $\mathcal{G}^{\gamma}$ is given by
$$\mathcal{G}^1=\frac{x_1^2}{2}+\frac{m-1}{12}x_2^4,
\quad
\mathcal{G}^2=x_1x_2+\frac{m-2}{6}x_2^3.$$ 
For the case with $m=2$, we have $G(2,1,2)=B_2$ and  the potential function is
$$\mathcal{F}=\frac{1}{2}x_1^2x_2+\frac{1}{60}x_2^5.$$
This agrees with \eqref{I2m} with $m=4$.  (Notice that $B_2$ coincides with $I_2(4)=G(4,4,2)$.)

%%%%%%%%%%%%%%%%
\subsubsection{$G(m,1,3)$}
%%%%%%%%%%%%%%%%
A set of good basic invariants is given by
\begin{equation*}
\begin{cases}
x_1=\frac{(\sqrt{3})^{3m}}{3m} 
\left(\bar{\sigma}_1-\frac{1}{3m}\bar{\sigma}_2\bar{\sigma}_3+\frac{3m+1}{54m^2}\bar{\sigma}_3^3 
\right)\\
x_2=\frac{(\sqrt{3})^{2m}}{-3m} 
\left(\bar{\sigma}_2-\frac{m+1}{6m}\bar{\sigma}_3^2\right)\\
x_3=\frac{(\sqrt{3})^{m}}{3m} 
\bar{\sigma}_3
\end{cases}.
\end{equation*}
This agrees with the flat coordinates given in \cite[\S 7.6]{Arsie-Lorenzoni2016}.
The potential vector field $\mathcal{G}^{\gamma}$ is given as follows.
\begin{eqnarray*}
\mathcal{G}^1&=&
\frac{x_1^2}{2}+\frac{x_2^3}{6}+\frac{m-1}{2}x_2^2x_3^2+
\frac{(m-1)(m-2)}{8}x_2x_3^4+\frac{(m-1)(m^2-2m+2)}{60}x_3^6,\\
\mathcal{G}^2&=&x_1x_2+\frac{m(m-1)}{6}x_2x_3^3+\frac{m-1}{2}x_2^2x_3+
\frac{(m-1)(m-2)(m+1)}{40}x_3^5,\\
\mathcal{G}^3&=&x_1x_3+\frac{x_2^2}{2}+\frac{m-2}{2}x_2x_3^2+\frac{(m-2)(m-3)}{24}x_3^4.
\end{eqnarray*}

For the case with $m=2$, we have $G(2,1,3)=B_3$
 and  the potential function is
$$\mathcal{F}=\frac{1}{2}x_1^2x_3+\frac{1}{2}x_1x_2^2+\frac{1}{6}x_2^2x_3^3
+\frac{1}{6}x_2^3x_3+\frac{1}{210}x_3^7.$$
This agrees with eq.(2.47) in \cite{Dubrovin1998}.

%%%%%%%%%%%%%%%%
\subsubsection{$G(m,1,4)$}
%%%%%%%%%%%%%%%%
A set of good basic invariants is given by
\begin{equation*}
\begin{cases}
x_1=\frac{(\sqrt{4})^{4m}}{-4m} 
\left(\bar{\sigma}_1-\frac{1}{4m}\bar{\sigma}_2\bar{\sigma}_4-\frac{1}{8m}\bar{\sigma}_3^2+
\frac{4m+1}{32m^2}\bar{\sigma}_3\bar{\sigma}_4^2 -\frac{32m^2+12m+1}{1536m^3}\bar{\sigma}_4^4
\right)\\
x_2=\frac{(\sqrt{4})^{3m}}{4m} 
\left(\bar{\sigma}_2-\frac{m+1}{4m}\bar{\sigma}_3\bar{\sigma}_4+\frac{(m+1)(5m+1)}{96m^2}\bar{\sigma}_4^3 
\right)\\
x_3=\frac{(\sqrt{4})^{2m}}{-4m} 
\left(\bar{\sigma}_3-\frac{2m+1}{8m}\bar{\sigma}_4^2\right)\\
x_4=\frac{(\sqrt{4})^{m}}{4m} 
\bar{\sigma}_4
\end{cases}.
\end{equation*}
The computation of the potential vector field in this case is left to the interested readers.

\begin{remark}
Observe that a set of good basic invariants of $G(m,1,2)$ is obtained 
from that of $G(m, m,3)$ by setting $\sigma_3$ to  zero, 
and the potential vector field for $G(m, 1,2)$ is obtained from that 
of $G(m, m,3)$ by setting $x_3$ to   zero.
Similar relation holds for $G(m, 1,3)$ and $G(m, m, 4)$.
In a subsequent paper \cite{KM2023}, we shall show that, for any $n$, 
a set of good basic invariants and the potential vector field
of $G(m,1,n)$ is obtained from those of $G(m, m, n+1)$  in the same way.
\end{remark}


\begin{thebibliography}{10}
\bibitem{Arsie-Lorenzoni2016}
Arsie, Alessandro and  Lorenzoni, Paolo,
{\it Complex reflection groups, logarithmic connections and bi-flat $F$-manifolds},
Lett. Math. Phys. {\bf 107} (2017), no.10, 1919--1961. 

\bibitem{Bessis2015}
Bessis, David, 
{\it Finite complex reflection arrangements are $K(\pi, 1)$},
Ann. of Math. (2) {\bf 181} (2015), no. 3, 809--904.

\bibitem{Dubrovin1993-2}
Dubrovin, Boris,
{\it Geometry of 2D topological field theories}, 
in Integrable systems and quantum groups (Montecatini Terme, 1993), 120--348, 
Lecture Notes in Math. {1620}, Springer, Berlin, 1996. 


\bibitem{Dubrovin1998}
Dubrovin, Boris,
{\it Differential geometry of the space of orbits of a Coxeter group},  
Surveys in differential geometry: integrable systems, 181--211, 
Surv. Differ. Geom., 4, Int. Press, Boston, MA, 1998. 

%\bibitem{Dubrovin2004}
%Dubrovin, Boris,
%{\it On almost duality for Frobenius manifolds}.
%Geometry, topology, and mathematical physics, 75--132,
%Amer. Math. Soc. Transl. Ser. 2, 212, Adv. Math. Sci., 55, Amer. Math. Soc., Providence, RI, 2004.

%\bibitem{DubrovinZhang1998}
%Dubrovin, Boris; Zhang, Youjin,
%Extended affine Weyl groups and Frobenius manifolds. Compositio Math. 111 (1998), no. 2, 167--219. 
%%\bibitem{Iwasaki1997}
%Katsunori, Iwasaki,
%Basic Invariants of Finite Reflection Groups.
%Journal of Algebra 

\bibitem{Kato2017}
Kato, Mituso, 
{\it Flat basic invariants for complex reflection groups},
a talk given at Encounter with Mathematics, Tokyo,  June 24th, 2017.

\bibitem{KatoManoSekiguchi2015}
Kato, Mitsuo; Mano, Toshiyuki; Sekiguchi, Jiro,
{\it  Flat structure on the space of isomonodromic deformations}, SIGMA Symmetry Integrability Geom. Methods Appl. {\bf 16} (2020), Paper No. 110, 36 pp.

\bibitem{KMS2018}
Konishi, Yukiko; Minabe, Satoshi; Shiraishi, Yuuki,
{\it Almost duality for Saito structure and complex reflection groups}, 
J. Integrable Syst. {\bf 3} (2018), no. 1, xyy003, 48 pp. 

\bibitem{KM2020}
Konishi, Yukiko; Minabe, Satoshi,
{\it Almost duality for Saito structure and complex reflection groups II: the case of Coxeter and Shephard groups},
Pure Appl. Math. Q. {\bf 16} (2020), no.3, 721--754.

\bibitem{KM2023}
Konishi, Yukiko; Minabe, Satoshi,
{\it A reduction theorem for good basic invariants of finite complex reflection groups}, 
J. Algebra {\bf 677} (2025), 327--359.

\bibitem{LehrerTaylor}
Lehrer, Gustav I. and Taylor, Donald E.,
{\it Unitary reflection groups},
Australian Mathematical Society Lecture Series, 20. Cambridge University Press, Cambridge, 2009. viii+294 pp.


%\bibitem{Mano2015}
%Tomoyuki Mano,
%{\it Generalization of Okubo type ordinary differential equation
%to several variables and flat structure}, 
%a talk in 
%MSJ Autumn meeting 2015.

%\bibitem{NT2014}
%Norihiro Nakashima and Shuhei Tsujie,
%A canonical system of basic invariants of a finite reflection group,
%J. Algebra, Volume 406, pp. 143-153, May 2014. doi:10.1016/j.jalgebra.2014.02.012   arXiv:1211.6026
%
%\bibitem{NTT2016}
%Norihiro Nakashima, Hiroaki Terao, and Shuhei Tsujie,
%Canonical systems of basic invariants for unitary reflection groups,
%Canad. Math. Bull. Volume 59, No. 3, pp. 617-623, Sep. 2016. dx.doi.org/10.4153/CMB-2016-031-7   arXiv:1310.0570

\bibitem{OrlikSolomon}
Orlik, Peter; Solomon, Louis,
{\it Unitary reflection groups and cohomology}. 
Invent. Math. {\bf 59} (1980), no. 1, 77--94.
 
\bibitem{OrlikTerao} 
Orlik, Peter; Terao, Hiroaki,
{\it Arrangements of hyperplanes}. Grundlehren der mathematischen Wissenschaften,
300. Springer-Verlag, Berlin, 1992. xviii+325 pp. %ISBN: 3-540-55259-6 

\bibitem{Sabbah} 
Sabbah, Claude, 
{\it D\'eformations isomonodromiques et vari\'et\'es de Frobenius},
EDP Sciences, Les Ulis; CNRS \'{E}ditions, Paris, 2002. xvi+289 pp. 


\bibitem{SaitoSekiguchiYano}
Saito, Kyoji; Yano, Tamaki; Sekiguchi, Jiro,
On a certain generator system of the ring of invariants of a finite complex reflection group. 
Comm. Algebra {\bf 8} (1980), no. 4, 373--408. 

\bibitem{Saito1993}
Saito, Kyoji, 
{\it On a linear structure of the quotient variety by a finite reflexion group},
Publ. Res. Inst. Math. Sci. {\bf 29} (1993), no. 4, 535--579.
(Preprint version: RIMS-288, Kyoto Univ., Kyoto, 1979.)


%\bibitem{Satake2010}
%Satake, Ikuo Frobenius manifolds for elliptic root systems. Osaka J. Math. 47 (2010), no. 1, 301--330.

\bibitem{Satake2019}
Satake, Ikuo, {\it Coxeter transformation and the Frobenius structure},
a talk in the conference ``Mirror Symmetry and Related Topics, 2019'',
Dec. 9th-13th 2019,
held at Kyoto university.

\bibitem{Satake2020}
Satake, Ikuo, 
{\it Good Basic Invariants and Frobenius Structures}, 
Osaka J. Math. (In Press), arXiv:2004.01871v1.


%\bibitem{Satake2020elliptic}
%{\it Good Basic Invariants for Elliptic Weyl Groups and Frobenius Structures}.
%arXiv:2004.03587.

\bibitem{Springer}
Springer, T. A., {\it Regular elements of finite reflection groups}, Invent. Math. {\bf 25} (1974), 
159--198.

\end{thebibliography}
\end{document}